\documentclass{amsart}
\usepackage{stmaryrd}
\usepackage{silence}
\usepackage{upgreek}
\usepackage{amssymb}
\usepackage{faktor}
\usepackage[all,cmtip]{xy}
\usepackage{amsmath} 
\usepackage{mathrsfs}
\usepackage{tikz}
\usepackage{tikz-cd}
\usepackage{enumitem}
\usepackage{mathtools}
\usepackage[mathcal]{euscript}
\usepackage{graphicx}
\usepackage{url}
\usepackage{hyperref}
\usepackage[T1]{fontenc}
\usepackage{bbm}

\usetikzlibrary{shapes.geometric}
\usetikzlibrary{decorations.markings}
\usetikzlibrary{patterns,decorations.pathreplacing}

\hypersetup{colorlinks=true,linkcolor=blue}

\usepackage{ragged2e}

\newcommand{\id}{\mathrm{id}}

\newcommand{\A}{\mathcal{A}}

\newcommand{\mul}{\mathrm{mul}}

\definecolor{coloryellow}{RGB}{240,228,66}
\definecolor{colorskyblue}{RGB}{86,180,233}
\definecolor{colorvermillion}{RGB}{213,94,0}

\DeclareMathOperator{\coker}{coker}
\DeclareMathOperator{\Hom}{Hom}
\DeclareMathOperator{\rk}{rk}

\newcommand{\graphfont}{\mathsf}

\newcommand{\stargraph}[1]{\graphfont{S}_{#1}}
\newcommand{\graf}{\graphfont{\Gamma}}

\DeclareSymbolFont{sfletters}{OT1}{cmss}{m}{n}
\DeclareMathSymbol{\sTheta}{\mathord}{sfletters}{"02}

\theoremstyle{definition}
\newtheorem{definition}{Definition}[section]

\newtheorem{example}[definition]{Example}

\newtheorem{construction}[definition]{Construction}

\theoremstyle{plain}
\newtheorem{proposition}[definition]{Proposition}
\newtheorem{lemma}[definition]{Lemma}
\newtheorem{corollary}[definition]{Corollary}
\newtheorem{theorem}[definition]{Theorem}
\newtheorem{conjecture}[definition]{Conjecture}

\theoremstyle{remark}
\newtheorem{remark}[definition]{Remark}

\makeatletter
\@addtoreset{definition}{section}
\makeatother

\usepackage{marginnote}
    \DeclareFontFamily{U}{wncy}{}
    \DeclareFontShape{U}{wncy}{m}{n}{<->wncyr10}{}
    \DeclareSymbolFont{mcy}{U}{wncy}{m}{n}
    \DeclareMathSymbol{\Sha}{\mathord}{mcy}{"58}

\newsavebox{\foobox}

\title{Representation asymptotics in the homology of pure graph braid groups}
\author{Louis Hainaut}
\author{Ben Knudsen}
\author{Nicholas Wawrykow}
\begin{document}
\begin{abstract}
We give explicit formulas for the asymptotic Betti numbers, over an arbitrary field, of the ordered configuration spaces of a graph. In characteristic zero, we further give explicit formulas for the asymptotic multiplicities in homology of many irreducible representations of the symmetric group, in the spirit of representation stability.
\end{abstract}
\maketitle
\section{Introduction}

The study of configuration spaces has been a core concern of mathematics for more than half a century, resulting in a myriad of applications to and interconnections among a who's who of disciplines. Over the past two decades, with initial motivation stemming from robotics and geometric group theory \cite{Abrams:CSBGG, Swiatkowski:EHDCSG,Ghrist:CSBGGR}, increasing attention has been paid to configuration spaces of graphs. In recent years, the homology of these aspherical spaces---hence ``graph braid group''---has been an object of particularly intense focus, with substantial progress in the unordered setting.\footnote{See \cite{AnDrummond-ColeKnudsen:AHGBG} and the references therein.} In contrast, computational knowledge in the case of ordered configuration spaces---equivalently, pure graph braid groups---is almost nil. Our object here is to address this gap.

To state our results, fix a field $\mathbb{F}$, and write $\mathcal{H}_i(\graf)_k=H_i(F_k(\graf);\mathbb{F})$, where $F_k(\graf)$ denotes the configuration space of $k$ ordered particles in the graph $\graf$, i.e., the complement in $\graf^k$ of the union of the diagonals. We also recall that the irreducible representations of the symmetric group $\Sigma_k$ in characteristic zero are the Specht modules $V_\mu$, where $\mu$ is a partition of $k$, and that a partition $\lambda$ of a smaller number may be stabilized or ``padded'' to a partition $\lambda[k]$ of $k$---see Section \ref{section:functions of partitions} for details. We write $\mul_\lambda\,\mathcal{H}_i(\graf)$ for the function sending $k$ to the multiplicity of $V_{\lambda[k]}$ in $\mathcal{H}_i(\graf)_k$. 

Our results will be formulated in terms of numbers $\Lambda^S_\graf$, $\Delta^S_\graf$, and $\mathcal{E}_\graf^S$, certain counts of connected components of $\graf_S$, the graph obtained by exploding a set $S$ of essential (valence or degree $d(v)$ at least three) vertices---see Section \ref{section:configuration spaces of graphs} for details. The maximum of each such number over all $S$ of fixed cardinality $i$ is a graph invariant, which we denote by $\Lambda^i_\graf$, and so forth.

\begin{theorem}\label{thm:conceptual asymptotics}
If $\graf$ is a connected graph not homeomorphic to the circle such that $\Lambda^i_\graf>0$, then we have
\[\dim \mathcal{H}_i(\graf)\approx\sum_{S}\left[\bigsqcap_{w\in S}\left(d(w)-2\right)\right]\dim\mathcal{H}_0(\graf_S),\]
where $S$ ranges over sets of essential vertices of cardinality $i$, and $\approx$ denotes equality modulo $o(k^{\Lambda^i_\graf-1}k!)$. In characteristic zero, the same conclusion holds after substituting $\Delta^i_\graf>1$, $\mul_\lambda$, and $o(k^{|\lambda|+\Delta^i_\graf-1})$.
\end{theorem}

This theorem reduces the asymptotic study of the homology of ordered configuration spaces of graphs to a question about connected components. For Betti numbers, this problem is essentially a counting problem, which we solve completely.

\begin{theorem}\label{thm:asymptotics}
If $\graf$ is a connected graph such that $\Lambda^i_\graf>0$, then we have
\[\dim \mathcal{H}_i(\graf)\approx \left[\sum_S\frac{e^{\mathcal{E}^S_\graf}}{(\Lambda^i_\graf-1)!}\bigsqcap_{w\in S} (d(w)-2)\right]k^{\Lambda^i_\graf-1}k!,\] where $S$ ranges over sets of essential vertices of cardinality $i$ such that $\Lambda^S_\graf=\Lambda^i_\graf$, and $\approx$ denotes equality modulo $o(k^{\Lambda^i_\graf-1}k!)$.
\end{theorem}

In the characteristic zero case, although we are able to calculate $\mul_\lambda\,\mathcal{H}_0(\graf_S)$ completely, some assumption is needed in order to make a clean identification of the terms with the highest order of growth.

\begin{theorem}\label{thm:rep asymptotics}
Let $\graf$ be a connected graph with $\Delta^i_\graf>1$ and $\lambda$ a partition, and suppose that either $\lambda=\varnothing$ or that $\Lambda^S_\graf>0$ whenever $\Delta^S_\graf=\Delta^i_\graf$. In characteristic zero, we have
{
\[\mul_\lambda\, \mathcal{H}_i(\graf)\approx \left[\sum_{S}\frac{\dim V_\lambda}{(|\lambda|+\Delta_\graf^i-1)!}\binom{|\lambda|+\Lambda_\graf^S-1}{\Lambda_\graf^S-1} \bigsqcap_{w\in S} (d(w)-2)\right]k^{|\lambda|+\Delta^i_\graf-1},\] 
}where $S$ ranges over sets of essential vertices of cardinality $i$ such that $\Delta^S_\graf=\Delta^i_\graf$, and $\approx$ denotes equality modulo $o(k^{|\lambda|+\Delta^i_\graf-1})$.  The same formula holds if instead $|\lambda|\geq\mathcal{E}^i_\graf$, with the summation taken over $S$ such that $\Delta^S_\graf=\Delta^i_\graf$ and $\Lambda_\graf^S>0$.
\end{theorem}

\begin{remark}
In Theorem \ref{thm:rep asymptotics}, the case $\lambda=\varnothing$ of the trivial representation is the main result of \cite{AnDrummond-ColeKnudsen:AHGBG}, but the argument here is novel and arguably more direct. In all other cases, the result is new; indeed, calculations of this type are exceedingly rare in the study of configuration spaces---see \cite[Prob. 3.5]{Farb:RS}.
\end{remark}

\begin{remark}\label{remark:generators}
The factors of $d(w)-2$ in our theorems are tied to the topology of the star graph $\stargraph{n}$ of valence $n$, suggesting the existence as a summand of $\mathcal{H}_1(\stargraph{n})$ of a free module of rank $n-2$ over the twisted algebra $\mathcal{A}_n$ (see Section \ref{section:twisted algebra}), whose asymptotics match the overall asymptotics. Preliminary calculations suggest that we have identified the generators of this module, at least for $n\geq 4$; for $n=3$, complications arise from the fact that $\mathcal{H}_1(\stargraph{3})$ is not finitely presented over $\mathcal{A}_3$ \cite{Wawrykow:HGROCSSG}.
\end{remark}

\begin{remark}
In view of the K\"{u}nneth theorem the restriction to connected graphs is hardly a restriction (see Example \ref{example:disjoint union} below). The exclusion of the circle in Theorem \ref{thm:conceptual asymptotics} results not from a breakdown of the method but rather from a breakdown of the definitions of the graph invariants involved. In any case, the configuration spaces of the circle are well understood---see \cite[p. 292]{CrabbJames:FHT}, for example. The inequalities $\Lambda^i_\graf>0$ and $\Delta^i_\graf>1$ hold except possibly in the case $i=1$.\footnote{See \cite{KoPark:CGBG} for partial results in degree $1$.}. 
\end{remark}

\subsection{Torsion} The observant reader will have noticed that the asymptotics asserted in Theorem \ref{thm:asymptotics} are independent of the characteristic of the background field. This fact is directly related to the question of torsion in the homology of pure graph braid groups, of which none is known to exist. This scarcity has led to the following belief, which, although widespread, has very little support in theory or computation.\footnote{See \cite{KoPark:CGBG} for the case of $H_1$ and \cite{ChettihLuetgehetmann:HCSTL} for the case of trees with loops.}

\begin{conjecture}[Folklore]\label{conjecture:torsion}
For any graph $\graf$ and $k\geq0$, the space $F_k(\graf)$ has torsion-free homology. 
\end{conjecture}

We submit that the characteristic independence of Theorem \ref{thm:asymptotics} should be interpreted as the assertion that \emph{this conjecture holds asymptotically}. More precisely, it implies that any torsion must grow slowly as a function of $k$ relative to the overall rate of growth.

\begin{corollary}
Let $\graf$ be a graph with no cycle components and $\Lambda^i_\graf>0$. For any prime $p$, the order of the $p$-torsion subgroup of $H_i(F_k(\graf);\mathbb{Z})$ lies in $o(p^{k^{\Lambda^i_\graf-1}k!})$.
\end{corollary}

This result underscores how far we are from an answer to the question posed by Conjecture \ref{conjecture:torsion}, for, although it leaves room for phenomenally large amounts of torsion, it is nevertheless the only general, quantitative result of its kind.

\subsection{Universal generators} A second important theme in the study of configuration spaces of graphs is that of universal generators and relations. We say that a collection of graphs $\mathcal{G}$ is a set of universal generators (in degree $i$) if, for every graph $\graf$, the group $H_i(F_k(\graf);\mathbb{Z})$ is generated by homology classes arising from topological subgraphs homeomorphic to members of $\mathcal{G}$.

\begin{conjecture}\label{conjecture:universal finite generation}
A finite set of universal generators exists in every degree.\footnote{See \cite{KoPark:CGBG, ChettihLuetgehetmann:HCSTL, MaciazekSawicki:NAQSG, AnKnudsen:OSHPGBG, KnudsenRamos:RCUFGHGBG} for partial results.}
\end{conjecture} 

In this context, Theorem \ref{thm:asymptotics} should be interpreted as identifying a set of \emph{asymptotic} universal generators, namely the set of disjoint unions of stars.

\begin{corollary}
Writing $H_i(F_k(\graf);\mathbb{Z})_0$ for the subgroup of homology generated by classes arising from topological subgraphs homeomorphic to an $i$-fold disjoint union of star graphs, we have
\[\lim_{k\to \infty} \frac{\rk H_i(F_k(\graf);\mathbb{Z})}{\rk H_i(F_k(\graf);\mathbb{Z})_0}=1\]
\end{corollary}

With more effort, we believe that we can refine this set of generators to a finite list, establishing the asymptotic validity of Conjecture \ref{conjecture:universal finite generation}---see Remark \ref{remark:generators}.

\subsection{Stability phenomena} Our main results lie in the venerable tradition of the study of (homological) stability phenomena. Since the pioneering work of \cite{ChurchEllenbergFarb:FIMSRSG} on representation stability, a standard strategy has emerged in this area. 
\begin{enumerate}
\item Produce an action of an algebraic object $\mathcal{A}$ on the homology of interest.
\item Show that the homology of interest arises as a subquotient of a finitely generated $\mathcal{A}$-module.
\item Show that finitely generated $\mathcal{A}$-modules have the desired growth behavior.
\item Show that $\mathcal{A}$ is Noetherian.
\end{enumerate}
\noindent The first two steps, which are typically the easy steps in an argument of this kind, go through without issue in our setting, although we will not explain quite how. Unfortunately, the third step exceeds at least our abilities of analysis, and the fourth step fails fundamentally, since the algebraic object that arises is non-Noetherian \cite{Wawrykow:HGROCSSG}.

That we are nevertheless able to characterize the relevant asymptotics is a remarkable feature of this work, but it does leave open the question of the precise eventual behavior of the Betti numbers and multiplicities. Regarding this question, we formulate the following.

\begin{conjecture}
For any partition $\lambda$, the function $\mul_\lambda\,\mathcal{H}_i(\graf)$ is eventually equal to a polynomial.
\end{conjecture}

If the conjecture holds, then Theorem \ref{thm:rep asymptotics} amounts to the calculation of the degree and leading coefficient of this polynomial. By \cite{AnDrummond-ColeKnudsen:ESHGBG,AnDrummond-ColeKnudsen:AHGBG}, the conjecture does hold in the case $\lambda=\varnothing$ of the trivial representation. We have no corresponding conjecture for the behavior of the Betti numbers; suffice it to say that they are \emph{not} eventually polynomial times factorial.

\subsection{Outline of argument} Conceptually, our argument has two main steps. The first is a geometric reduction appealing to the \emph{vertex explosion spectral sequence} of Proposition \ref{prop:vess}, which has the form \[E^2_{p,q}\cong \mathrm{Tor}_p^{\mathcal{A}_B}\left(\mathcal{H}_q(\graf_B)\right)\implies \mathcal{H}_{p+q}(\graf),\] where $\graf_B$ is a disjoint union of stars and intervals resulting from the explosion of a collection of bivalent vertices $B$, and $\mathcal{A}_B$ is a certain twisted algebra identified in Section \ref{section:twisted algebra}. 

The leading terms appearing in the statements of our main results arise from the asymptotics of the $p=0$ column of this $E^2$-page; therefore, the second step of the argument consists in calculating these asymptotics and bounding the growth of the remaining terms. The basic approach to both tasks is induction on the degree of homology, which is facilitated through an exact sequence relating homology in different degrees in the case of a star graph (see Proposition \ref{prop:star explosion}). In practice, these two conceptual steps are of quite different technical weight, with the bulk of the work going toward the inductive argument.

In linear terms, the first three sections of the paper following the introduction are concerned with foundational material in three distinct areas: Section \ref{section:twisted algebra} lays the algebraic foundations with a rudimentary development of the theory of symmetric sequences and twisted algebras; Section \ref{section:configuration spaces of graphs} supplies the necessary topological information concerning configuration spaces of graphs, in particular the {\'{S}}wi\k{a}tkowski complex and its derivatives; and Section \ref{section:asymptotics} develops the required techniques for dealing with the asymptotics of the functions of interest. Following this background, Section \ref{section:toward the main results} proves the main results assuming a certain technical estimate, which is established in Section \ref{section:technical estimate}.

\subsection{Conventions}\label{section:conventions} We work over a fixed ground field $\mathbb{F}$. All vector spaces are vector spaces over $\mathbb{F}$, and all homology is taken with coefficients in $\mathbb{F}$. The phrase ``in characteristic zero'' refers to the situational assumption that $\mathrm{char}(\mathbb{F})=0$.

A partition is a tuple $\lambda=(\lambda_1,\ldots, \lambda_m)$ of integers with $\lambda_1\geq \lambda_2\geq \cdots \geq\lambda_m>0$. We write $|\lambda|$ for the sum of the entries of $\lambda$. We write $\Pi$ for the set of all partitions.

When dealing with binomial coefficients, we adopt the convention that $\binom{n}{n}=1$ for $n<0$.

Modules are right modules unless otherwise specified.

\subsection{Acknowledgements} The first author was supported by the Knut and Alice Wallenberg Foundation through grant no. 2023.0446. The second author benefited from conversations with Byunghee An, Gabriel Drummond-Cole, and Eric Ramos, as well as the support of NSF grant DMS 2551600.

\section{Twisted algebra}\label{section:twisted algebra}

As indicated above, our strategy will begin by appealing to a certain spectral sequence, which will reduce the main theorems to the estimation of certain Tor terms calculated in the setting of symmetric sequences. In this section, we lay out a basic algebraic framework for such calculations. For a more leisurely and expansive account of these ideas, we direct the reader to \cite{SamSnowden:ITCA}.

\subsection{Basic notions} We begin by recalling the definition of a symmetric sequence, which is the fundamental unit in our study.

\begin{definition}
A \emph{symmetric sequence} is a collection $\mathcal{X}=\left(\mathcal{X}_k\right)_{k\geq0}$, where $\mathcal{X}_k$ is a (graded) vector space equipped with a (right) action of $\Sigma_k$.  The \emph{tensor product} of symmetric sequences $\mathcal{X}$ and $\mathcal{Y}$ is the symmetric sequence defined by
\[(\mathcal{X}\otimes\mathcal{Y})_k=\bigoplus_{i+j=k}\mathrm{Ind}_{\Sigma_i\times\Sigma_j}^{\Sigma_k}\mathcal{X}_i\boxtimes\mathcal{Y}_j.\]
\end{definition}

In other words, a symmetric sequence is simply a functor from the groupoid of finite sets and bijections. From this perspective, the tensor product of symmetric sequences is an instance of the more general construction of Day convolution. When discussing symmetric sequences, we will refer to the parameter $k$ as the \emph{weight.}\footnote{Elsewhere in the literature, objects equivalent to symmetric sequences are referred to variously as $\mathrm{FB}$-\emph{modules} and \emph{species}, and a synonym for weight is \emph{arity}.}

\begin{example}\label{example:disjoint union}
The symmetric sequence of greatest relevance here is $H_*(F(X))_k:=H_*(F_k(X))$. The tensor product acquires topological significance in this setting via the canonical isomorphism
\[H_*(F(X\sqcup Y))\cong H_*(F(X))\otimes H_*(F(Y)).\]
\end{example}

\begin{example}
A (graded) vector space $V$ determines two distinct symmetric sequences. The first, concentrated in weight $0$, we denote (abusively) by $V$; the second, concentrated in weight $1$, we denote by $V(1)$.
\end{example}

\begin{definition}
The $n$th \emph{shift} of the symmetric sequence $\mathcal{X}$ is the symmetric sequence $\mathcal{X}(n)=\mathcal{X}\otimes \mathbb{F}(1)^{\otimes n}$.
\end{definition}

Symmetric sequences form a category in which a morphism is simply a collection of equivariant maps. The tensor product of symmetric sequences endows this category with a symmetric monoidal structure, in which the symmetry and associator are inherited in a straightforward way from those of the tensor product of graded vector spaces, and where the unit is the ground field concentrated in weight $0$.

\begin{definition}
A \emph{twisted algebra} is a monoid in the symmetric monoidal category of symmetric sequences.
\end{definition}

In concrete terms, a twisted algebra is a graded algebra, together with an action of $\Sigma_k$ on the weight $k$ summand, such that multiplication of weight $i$ elements with weight $j$ elements is $\Sigma_i\times\Sigma_j$-equivariant.

\begin{example}\label{example:free algebras}
As in the classical setting, the free twisted algebra generated by the symmetric sequence $\mathcal{X}$ is the twisted tensor algebra $\mathcal{T}(\mathcal{X})=\bigoplus_{d\geq0}\mathcal{X}^{\otimes d}$, with the monoid structure induced by the identifications $\mathcal{X}^{\otimes d_1}\otimes\mathcal{X}^{\otimes d_2}\cong \mathcal{X}^{\otimes (d_1+d_2)}$. The free twisted commutative algebra is $\mathcal{S}(\mathcal{X})=\bigoplus_{d\geq0}\mathcal{X}^{\otimes d}_{\Sigma_d}$. In the case $\mathcal{X}=\mathbb{F}^S(1)$ for a set $S$, we write $\mathcal{S}_S$, or simply $\mathcal{S}_n$ if $S=\{1,\ldots, n\}$ (resp. $\mathcal{T}_S$, $\mathcal{T}_n$).
\end{example}

Commutativity is a somewhat subtle phenomenon in the twisted setting. A quick calculation of tensor products shows that $(\mathcal{T}_1)_k\cong\mathbb{F}[\Sigma_k]$, while $(\mathcal{S}_1)_k\cong\mathbb{F}$, which is a stark contrast with the classical accident that the tensor and polynomial algebras on one generator are isomorphic. In other words, an element in a twisted algebra may not only fail to commute with other elements; it may fail to commute \emph{with itself}. This dichotomy of intercommutativity and autocommutativity, which are independent conditions, is a fundamental complication of the twisted setting, one extreme of which is exhibited by the following twisted algebra, which will be of crucial importance for us.

%Unpacking the definitions, commutativity for a twisted algebra amounts to the condition that the results of multiplying a weight $k$ element with a weight $\ell$ element in the two possible orders differs by the block permutation in $\Sigma_{k+\ell}$ interchanging the first $k$ and last $\ell$ elements.

\begin{example}
We write $\mathcal{A}_S$ for the twisted algebra freely generated by an $S$-indexed set of intercommuting but not autocommuting weight $1$ generators, i.e., $\mathcal{A}_S=\mathcal{T}_1^{\otimes S}$. As above, we write $\mathcal{A}_n$ in the case $S=\{1,\ldots,n\}$.
\end{example}

These algebras have important topological significance for us; indeed, as we observe in Proposition \ref{prop:H0 description} below, there is an isomorphism  $\mathcal{A}_n\cong \mathcal{H}_0\left(\sqcup_n [0,1]\right)$. Correspondingly, these algebras will act in our setting by adding new particles to the leaves of a graph---see Proposition \ref{prop:leaf stabilization}. To aid in identifying such actions, we close with the following simple criterion.

\begin{proposition}\label{prop:action criterion}
Let $\mathcal{X}$ be a symmetric sequence and $S$ a set. The structure of a right $\mathcal{A}_S$-module on $\mathcal{X}$ is equivalent to the data of a collection of maps 
$\varphi_{s,k}:\mathcal{X}_k\to \mathcal{X}_{k+1}$ for $s\in S$ and $k\geq0$, such that, for every $k\geq0$ and $s\neq t\in S$, 
\begin{enumerate}
\item $\varphi_{s,k}$ is $\Sigma_k$-equivariant, where we make the standard identification of $\Sigma_k$ as a subgroup of $\Sigma_{k+1}$, and
\item  the two composites in the diagram
\[\xymatrix{
\mathcal{X}_k\ar[rr]^-{\varphi_{s,k}}\ar[d]_-{\varphi_{t,k}}&&\mathcal{X}_{k+1}\ar[d]^-{\varphi_{t,k+1}}\\
\mathcal{X}_{k+1}\ar[rr]^-{\varphi_{s,k+1}}&&\mathcal{X}_{k+2}
}\] differ by the action of the transposition $\tau_{k+1,k+2}\in \Sigma_{k+2}$.
\end{enumerate}
\end{proposition}

\begin{remark}
Note that, in the case $s=t$ of autocommutativity, the diagram of Proposition \ref{prop:action criterion} automatically commutes, so the requirement is that the common composite is invariant under the action of $\tau_{k+1,k+2}$.
\end{remark}

\begin{remark}
The action of a given twisted algebra may often be interpreted as the action of a combinatorial category; for example, an $\mathcal{S}_1$-module is equivalent to a $\mathrm{FI}$-module in the sense of \cite{ChurchEllenbergFarb:FIMSRSG}. See \cite{Ramos:GRSFM} and \cite[Rmk. 2.8]{Wawrykow:HGROCSSG} for the categories corresponding to $\mathcal{S}_m$ and $\mathcal{A}_n$, respectively.
\end{remark}

\subsection{Some results on Tor} The reader will recall from the introduction that our larger argument rests on an analysis of a certain spectral sequence. In this section, we discuss the homological algebra we will use to identify and calculate the entries on its $E^2$-page, which turn out to be derived functors generalizing the classical Tor functors. The main results of the section are Corollaries \ref{cor:koszul tor} and \ref{cor:tor dimension}.

Given a twisted algebra $\mathcal{A}$, the category of left (or right) $\mathcal{A}$-modules forms an Abelian category. In particular, the usual maneuvers of homological algebra are valid in this context. Of particular interest for us will be the calculation of the derived functors $\mathrm{Tor}_p^{\mathcal{A}_S}(\mathcal{M})$ of the functor on the category of right $\mathcal{A}_S$-modules given by the relative tensor product $\mathcal{M}\mapsto \mathcal{M}\otimes_{\mathcal{A}_S}\mathbb{F}$, which is right exact. 

\begin{construction}\label{construction:koszul complex}
Given a set $S$, we define a complex $\mathcal{K}_S$ of left $\mathcal{A}_S$-modules as follows. First, in the case of a singleton, we define $\mathcal{K}_1$ to be the complex 
\[  \mathcal{A}_1\otimes \mathbb{F}(1)\to \mathcal{A}_1,\] where the map in question is the restriction of the multiplication homomorphism $\mathcal{A}_1\otimes\mathcal{A}_1\to \mathcal{A}_1$ to weight $1$ in the second factor. For general $S$, we then define $\mathcal{K}_S=\bigotimes_{s\in S} \mathcal{K}_1.$ 
\end{construction}

We refer to this complex as the (twisted) \emph{Koszul complex}. Via the augmentation of $\mathcal{A}_1$, i.e., the projection onto the weight $0$ component, the Koszul complex is augmented over the ground field in weight $0$.

\begin{proposition}\label{prop:koszul complex}
The augmentation $\mathcal{K}_S\to \mathbb{F}$ is a quasi-isomorphism.
\end{proposition}
\begin{proof}
The general case reduces to the finite case by the exactness of filtered colimits, thence to the case of a singleton by flatness. In this case, the claim is that the restricted multiplication map is an isomorphism in positive weights, which holds by construction.
\end{proof}

Thus, the Koszul complex for $S$ forms a free resolution of the ground field as a left $\mathcal{A}_S$-module.

\begin{remark}
The reader may find the appearance of this straightforward analogue of the Koszul complex suprising; indeed, the classical Koszul complex is a resolution of the ground field over a free \emph{commutative} algebra, and $\mathcal{A}_S$ is non-commutative! The resolution of this dissonance lies in the dichotomy of intercommutativity vs. autocommutativity discussed above. Classically, the latter condition is automatic, but here it must be ``resolved separately'' from intercommutativity. Correspondingly, there is no finite free resolution of the ground field over $\mathcal{S}_1$.
\end{remark}

\begin{corollary}\label{cor:koszul tor}
Let $\mathcal{M}$ be a right $\mathcal{A}_S$-module. There is a natural isomorphism
\[\mathrm{Tor}_p^{\mathcal{A}_S}(\mathcal{M})\cong H_p(\mathcal{M}\otimes_{\mathcal{A}_S}\mathcal{K}_S).\]
\end{corollary}

\begin{corollary}\label{cor:tor dimension}
For any right $\mathcal{A}_S$-module $\mathcal{M}$, we have $\mathrm{Tor}_p^{\mathcal{A}_S}(\mathcal{M})=0$ for $p>|S|$.
\end{corollary}

\begin{corollary}\label{cor:iterated tor}
For any right $\mathcal{A}_S$-module $\mathcal{M}$, $s\in S$, and $q\geq0$, there is a canonical two-step filtration of $\mathrm{Tor}_q^{\mathcal{A}_S}(\mathcal{M})$ with associated graded object
\[\mathrm{gr}_p\mathrm{Tor}_q^{\mathcal{A}_S}(\mathcal{M})\cong \begin{cases}
\mathrm{Tor}_1^{\mathcal{A}_{\{s\}}}\left(\mathrm{Tor}_{q-1}^{\mathcal{A}_{S\setminus\{s\}}}(\mathcal{M})\right)&\quad p=1\\
\mathrm{Tor}_0^{\mathcal{A}_{\{s\}}}\left(\mathrm{Tor}_q^{\mathcal{A}_{S\setminus\{s\}}}(\mathcal{M})\right)&\quad p=0
\end{cases}
\]
\end{corollary}
\begin{proof}
We view $\mathcal{M}\otimes_{\mathcal{A}_S}\mathcal{K}_S\cong \mathcal{M}\otimes_{\mathcal{A}_S}\mathcal{K}_{S\setminus\{s\}}\otimes\mathcal{K}_{\{s\}}$ as a double complex. One of the spectral sequences arising from this double complex has the bigraded object of the conclusion as its $E^2$-page, hence collapses for degree reasons, yielding the claim.
\end{proof}

%We close this section with a simple Tor calculation that will be of use below.
%
%\begin{lemma}\label{lem:difference tor}
%Viewing $\mathcal{A}_2$ as an $\mathcal{A}_1$-module via the difference of the two canonical actions, we have
%\[\mathrm{Tor}_p^{\mathcal{A}_1}(\mathcal{A}_2)\cong\begin{cases}
%\mathcal{A}_1&\quad p=0\\
%0&\quad \text{otherwise.}
%\end{cases}\]
%\end{lemma}
%\begin{proof}
%Consider the composition $\mathcal{A}_2(1)\to \mathcal{A}_2=\mathcal{A}_1\otimes\mathcal{A}_1\to \mathcal{A}_1$, where the first map is the differential in the complex $\mathcal{K}_1\otimes_{\mathcal{A}_1}\mathcal{A}_2$ and the second map is the algebra structure map of $\mathcal{A}_1$.
%\end{proof}

%We close by recording for future reference the following standard fact about these derived functors.
%
%\begin{proposition}\label{prop:algebraic les}
%Let $0\to \mathcal{M}\to \mathcal{N}\to \mathcal{P}\to 0$ be an exact sequence of left $\mathcal{A}_S$-modules. There is canonical exact sequence of the following form:
%\[\cdots\to \mathrm{Tor}_p^{\mathcal{A}_S}(\mathcal{M})\to \mathrm{Tor}_p^{\mathcal{A}_S}(\mathcal{N})\to \mathrm{Tor}_p^{\mathcal{A}_S}(\mathcal{P})\to \mathrm{Tor}_{p-1}^{\mathcal{A}_S}(\mathcal{M})\to \cdots.\]
%\end{proposition}

\section{Configuration spaces of graphs}\label{section:configuration spaces of graphs}

This section discusses the topological input to our argument, which primarily takes the form of the {\'{S}}wi\k{a}tkowski complex, a wonderful tool for calculating the homology of configuration spaces of graphs. In fact, this tool will be rather underutilized in our argument, as we only require three of its derivative products, namely the exact sequence of Proposition \ref{prop:star explosion}, the spectral sequence of Proposition \ref{prop:vess}, and the simple description of connected components given in Proposition \ref{prop:H0 description}. Although these results could be derived by other means, the {\'{S}}wi\k{a}tkowski complex is probably the most transparent and direct method. 

\subsection{Graphs and leaf stabilization}

A graph is a finite CW complex of dimension $1$ (thus, we do not permit discrete components). The set of edges of the graph $\graf$ is denoted $E(\graf)$, or just $E$ when no confusion can arise. Similarly, we have the set $V$ of its vertices, the set $L\subseteq V$ of its leaves, i.e., univalent vertices, and the set $H$ of its half-edges. Given $v\in V$, we write $H(v)\subseteq H$ for the set of half-edges at $v$. 

\begin{example}
The star graph $\stargraph{n}$ is the cone on the discrete space $\{1,\ldots, n\}$ with its canonical cell structure (see Figure \ref{fig:star graphs}).
\end{example}

\begin{figure}
\includegraphics[scale=.5]{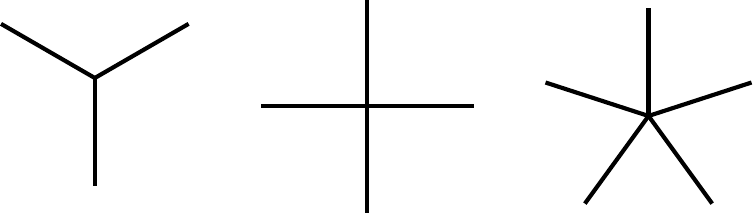}
\caption{The star graphs $\stargraph{3}$, $\stargraph{4}$, and $\stargraph{5}$.}
\label{fig:star graphs}
\end{figure}

Given $S\subseteq V$, we write $\graf_S$ for the graph obtained by exploding each of the vertices in $S$. Here, exploding a vertex refers to the procedure of removing it and compactifying each of the resulting open $1$-cells by the addition of a leaf (see Figure \ref{fig:explosion}). Thus, $\graf_v$ is homeomorphic to the complement in $\graf$ of any sufficiently small open neighborhood of $v$.

We now define certain connectivity invariants, which we will use to characterize the asymptotics of the Betti numbers and stable multiplicities of $\mathcal{H}_i(\graf)$.

\begin{definition}
Given a set $S$ of essential vertices of $\graf$, we write $\Delta_\graf^S=|\pi_0(\graf_S)|$. We write $\mathcal{E}^S_\graf$ for the number of components of $\graf_S$ containing an essential vertex and $\Lambda^S_\graf$ for the number containing no essential vertex. We write $\Delta_\graf^i$, $\Lambda_\graf^i$, and $\mathcal{E}_\graf^i$ for the maxima of the respective invariants over sets of essential vertices of cardinality $i$. By convention, we set $\Delta^i_\graf=\Delta^0_\graf$ for $i<0$, and so on.
\end{definition}

\begin{figure}[h]
\includegraphics[scale=.5]{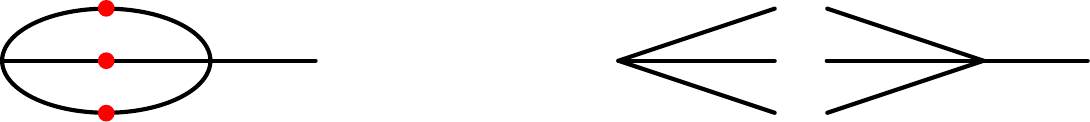}
\caption{A graph and its explosion at a set of three bivalent vertices.}
\label{fig:explosion}
\end{figure}

\begin{remark}
The importance of the invariant $\Delta_\graf^i$ in the context of graph braid groups was first indicated in \cite{Ramos:SPHTBG}. For this reason, it is sometimes known as the $i$th \emph{Ramos number} of $\graf$.
\end{remark}

Before moving on, we record two simple observations concerning these invariants.

\begin{lemma}\label{lem:explosion inequality}
For any set of essential vertices $S$ and $w\notin S$, we have $\Lambda^{S}_{\graf_w}\leq \Lambda^{S\cup\{w\}}_\graf$ (resp. $\Delta$). In particular, we have $\Lambda^{i-1}_{\graf_w}\leq\Lambda^i_\graf$ (resp. $\Delta$).
\end{lemma}
\begin{proof}
Both inequalities are immediate from the homeomorphism $(\graf_w)_S\cong\graf_{S\cup\{w\}}$.
\end{proof}

\begin{lemma}\label{lem:basic inequality}
If $\graf$ is a connected graph with at least $i>1$ essential vertices, then $\Lambda^{i-1}_{\graf}<\Lambda^i_{\graf}$ and $\Delta^{i-1}_{\graf}<\Delta^i_{\graf}$.
\end{lemma}
\begin{proof}
Assume without loss of generality that $\graf$ has no bivalent vertices, and fix a nonempty set $S$ of essential vertices. Connectedness guarantees that there is an essential vertex $v\notin S$ adjacent to some $w\in S$. By adjacency, exploding $\graf_S$ at $v$ has the effect of splitting the component of $\graf_S$ containing $v$ into a number of components strictly greater than one. One of these components is the edge between $v$ and $w$, which contains no essential vertex. We conclude that $\Lambda^{S\cup \{v\}}_\graf>\Lambda^S_\graf$ and $\Delta^{S\cup \{v\}}_\graf>\Delta^S_\graf$, implying the claim.
\end{proof}

We now discuss a well known stabilization mechanism for ordered configuration spaces of graphs, which amounts informally to adding a new particle to a leaf.\footnote{See \cite{Wawrykow:HGROCSSG} for a deep exploration of this structure in the case of a star graph and \cite{LuetgehetmannRecio-Mitter:TCCSFAGBG} for a more general and sophisticated stabilization mechanism.} In fact, although the induced algebraic structure at the level of homology will be crucial for our argument, the topological origin of this structure as leaf stabilization will play no role. We identify it here for the sake of completeness.

\begin{construction}\label{construction:leaf stabilization}
Fix a graph $\graf$. Given a leaf $v\in L$, we may choose a topological embedding $\varphi:\graf\sqcup[0,1]\to \graf$ sending $1$ to $v$ and restricting to the identity on the complement of the open star of $v$. We write $\sigma_{v,k}$ for the composite
\[F_k(\graf)\cong F_k(\graf)\times\{1\}\subseteq F_k(\graf)\times F_1([0,1])\subseteq F_{k+1}(\graf\sqcup[0,1])\xrightarrow{F_{k+1}(\varphi)}F_{k+1}(\graf),\] where the second inclusion is the inclusion of the subspace in which the first $k$ particles lie in $\graf$ and the last in $[0,1]$.
\end{construction}

\begin{proposition}\label{prop:leaf stabilization}
For every $i\geq0$, the maps $\{(\sigma_{v,k})_*\}_{v\in L(\graf), k\geq0}$ endow $\mathcal{H}_i(\graf)$ with a canonical $\mathcal{A}_{L}$-module structure.
\end{proposition}
\begin{proof}
For the module structure, we invoke Proposition \ref{prop:action criterion}; indeed, $\sigma_{v,k}$ is $\Sigma_k$-equivariant by construction, and the requirement that $\varphi$ restrict to the identity on the complement of the open star of $v$ guarantees that $\sigma_{w,k+1}\circ \sigma_{v,k}$ and $\sigma_{v,k+1}\circ\sigma_{w,k}$ differ by interchanging the $(k+1)$st and $(k+2)$nd particles, as required.

 For canonicity, any two choices of embedding $\varphi$ in Construction \ref{construction:leaf stabilization} are isotopic by an isotopy fixing $v$ and the complement of its open star, so $\sigma_{v,k}$ is unique up to homotopy. 
\end{proof}

%\begin{figure}
%\includegraphics[scale=.5]{sigmamap.pdf}
%\end{figure}

\begin{figure}
\includegraphics[scale=.5]{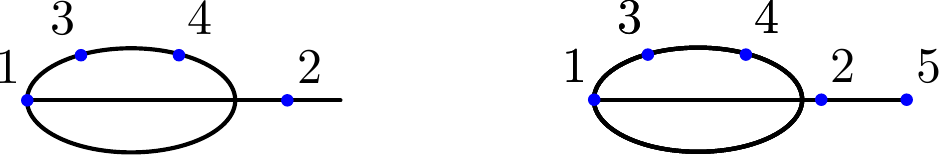}
\caption{A configuration of four particles and its image under leaf stabilization.}
\end{figure}

\subsection{{\'{S}}wi\k{a}tkowski complex} The goal of this section is to establish the following three topological results, all of which are well known, at least to experts. The reader who knows or is prepared to believe may wish to consider skipping over the proofs on a first reading.

Our first result is a special case of a more general ordered analogue of the vertex explosion exact sequence of \cite{AnDrummond-ColeKnudsen:SSGBG}, which we omit, since its generality serves no purpose here. The role of the result in our argument will be to facilitate an induction on the degree of homology.

\begin{proposition}\label{prop:star explosion}
There is a canonical exact sequence of $\Sigma_n$-equivariant $\mathcal{A}_n$-modules of the form
\[0\to \mathcal{H}_1(\stargraph{n})\to \mathcal{A}_n\otimes\delta_n(1)\to \mathcal{A}_n\to \mathcal{H}_0(\stargraph{n})\to 0,\] where $\delta_n\subseteq\mathbb{F}^n$ is the span of the differences of the standard basis vectors.
\end{proposition}

Note that, since $\delta_n\cong\mathbb{F}^{n-1}$ non-canonically, the $\mathcal{A}_n$-module $\mathcal{A}_n\otimes \delta_n(1)$ is free of rank $n-1$ in weight $1$.

The second result is likewise a special case of a more general ordered analogue of the type of spectral sequence described \cite{AnDrummond-ColeKnudsen:SSGBG}. It can also be viewed as an instance of the Mayer--Vietoris spectral sequence in factorization homology (in the stratified setting of \cite{AyalaFrancisTanaka:FHSS}). The analysis of the $E^2$-page of this spectral sequence will be the core of our argument.

\begin{proposition}[Vertex explosion spectral sequence]\label{prop:vess}
Let $\graf$ be a graph and $B$ a collection of bivalent vertices. There is a natural spectral sequence of $\mathcal{A}_{L(\graf)}$-modules of the form \[E^2_{p,q}\cong \mathrm{Tor}_p^{\mathcal{A}_B}\left(\mathcal{H}_q(\graf_B)\right)\implies \mathcal{H}_{p+q}(\graf),\] where $\mathcal{A}_B$ acts on $\mathcal{H}_q(\graf_B)$ via the homomorphism $\mathcal{A}_B\to \mathcal{A}_{L(\graphfont{\graf_B})}$ sending a bivalent vertex to the difference of the associated leaves.
\end{proposition}

The third and simplest result is a structured description of the connected components of the configuration spaces of a graph, which will serve as a base case for the induction on degree alluded to above.

\begin{proposition}\label{prop:H0 description}
Let $\graf$ be a graph with no cycle component such that $\mathcal{E}_\graf^0=m$ and $\Lambda^0_\graf=n$. There is a canonical isomophism 
\[\mathcal{H}_0(\graf)\cong\mathcal{S}_m\otimes \mathcal{A}_n\] of $\mathcal{A}_L$-modules, where a leaf acts via the generator indexed by its connected component. 
\end{proposition}

To establish these results, we recall a chain model due originally to {\'{S}}wi\k{a}tkowski \cite{Swiatkowski:EHDCSG}. The reader familiar with \cite{AnDrummond-ColeKnudsen:SSGBG,AnDrummond-ColeKnudsen:ESHGBG,AnDrummond-ColeKnudsen:AHGBG} will notice our recycling of terminology and notation from the unordered setting. This choice is made for the sake of simplicity and should cause no confusion.

As in the previous section, we write $\delta_S\subseteq \mathbb{F}^S$ for the span of the differences of the standard basis vectors.

\begin{construction}
First, fixing orientations of the edges of $\graf$, write $\partial_0(e)$ and $\partial_1(e)$ for the two half-edges of $e$. Define a function $\varepsilon:\mathbb{F}^{H}(1)\to \mathcal{A}_E\otimes \mathcal{A}_E^{\mathrm{op}}$ by the formula
\[
\varepsilon(h)=\begin{cases}
e(h)\otimes 1&\quad h=\partial_0(e(h))\\
1\otimes e(h)&\quad h=\partial_1(e(h)),
\end{cases}
\] where $e(h)$ is the edge associated to the half-edge $h$. Next, given $v\in V$, we write $\widetilde S(v)=\mathbb{F}\oplus \delta_{H(v)}(1)$; thus, $\widetilde S(v)$ is a symmetric sequence concentrated in weights $0$ and $1$, and we endow it with a grading by setting degree equal to weight. Using $\varepsilon$, we obtain a degree $-1$ endomorphism $\partial_v$ of $\widetilde S(v)\otimes \mathcal{A}_E$ as the composition
\[\widetilde S(v)\otimes \mathcal{A}_E\xrightarrow{(0,\epsilon)\otimes \id}(\mathcal{A}_E\otimes \mathcal{A}_E^\mathrm{op})\otimes \mathcal{A}_E\to \mathcal{A}_E\subseteq \widetilde S(v)\otimes \mathcal{A}_E,\] where the second arrow is the $\mathcal{A}_E$-bimodule structure of $\mathcal{A}_E$, and the inclusion is induced by the inclusion of the degree $0$ component of $\widetilde S(v)$.
\end{construction}

\begin{definition}
The (ordered, reduced) \emph{{\'{S}}wi\k{a}tkowski complex} is the differential graded symmetric sequence given by the tensor product \[\widetilde S(\graf)=\left(\bigotimes_{v\in V} \widetilde S(v)\right)\otimes \A_E,\] endowed with the differential induced by the $\partial_v$ for $v\in V$.
\end{definition}

\begin{remark}
For simplicity of exposition, we have used edge orientations in defining $\widetilde S(\graf)$, but different choices lead to canonically isomorphic complexes.
\end{remark}

By construction, the underlying symmetric sequence of this complex is canonically an $\mathcal{A}_E$-bimodule. The bulk of this structure is incompatible with the differential, but it is not difficult to see that the right action of a leaf is compatible provided we orient the leaf toward its univalent vertex (resp. left, away). Thus, we may regard $\widetilde S(\graf)$ as a differential graded $\mathcal{A}_L$-module.

\begin{theorem}[{\'{S}}wi\k{a}tkowski]\label{thm:swiatkowski}
There is a canonical isomorphism of $\mathcal{A}_{L}$-modules
\[
H_i(\widetilde S(\graf))\cong \mathcal{H}_i(\graf).
\]
\end{theorem}

\begin{remark}
Strictly speaking, this result does not appear in \cite{Swiatkowski:EHDCSG}, which deals only with the unordered case. The generalization to the ordered case was first recorded in \cite{Luetgehetmann:CSG}, which also corrected a mistake in the original work. Unfortunately, this work itself contained a mistake, which was subsequently corrected in \cite{ChettihLuetgehetmann:HCSTL}. None of these references deal with the claimed compatibility with leaf stabilization, but it is trivially verified by inspecting the arguments.
\end{remark}

With this complex in hand, we prove the three results stated above.

\begin{proof}[Proof of Proposition \ref{prop:star explosion}]
As a differential graded $\mathcal{A}_n$-module, $\widetilde S(\stargraph{n})$ is canonically isomorphic to $\mathcal{A}_n\otimes\delta_n(1)\to\mathcal{A}_n$, so the claim is immediate from Theorem \ref{thm:swiatkowski}.
\end{proof}

\begin{proof}[Proof of Proposition \ref{prop:vess}]
The spectral sequence in question arises by filtering the reduced {\'{S}}wi\k{a}tkowski complex of $\graf$ by the number of elements of $B$ involved in a basis element, i.e., by setting
\[F_p\,\widetilde{S}(\graf)=\sum_{S\subseteq B : |S|=p}\mathrm{im}\left(\widetilde S(\graf_{B\setminus S})\to \widetilde S(\graf)\right).\] Ignoring the differential, this filtration is split by the defining basis of $\widetilde S(\graf)$, and, by direct inspection, we have the isomorphism
\[E^0_{p,*}\cong\left( \bigoplus_{S\subseteq B:|S|=p}\mathbb{F}(1)^{\otimes S}\right) \otimes\widetilde S(\graf_B)\] of chain complexes for each $p\geq0$, and hence, invoking Theorem \ref{thm:swiatkowski}, the identification of the first page as
\[E^1_{*,q}\cong \left(\bigoplus_{S\subseteq B}\mathbb{F}(1)^{\otimes S}\right)\otimes\mathcal{H}_q(\graf_B)\cong \mathcal{K}_B\otimes_{\mathcal{A}_B} \mathcal{H}_q(\graf_B),\] where $\mathcal{K}_B$ is the twisted Koszul complex of Construction \ref{construction:koszul complex}. After inspecting the $d^1$ differential, the claim follows from Corollary \ref{cor:koszul tor}.
\end{proof}

\begin{proof}[Proof of Proposition \ref{prop:H0 description}]
By Example \ref{example:disjoint union} we may assume that $\graf$ is connected and show that $\mathcal{H}_0(\graf)$ is isomorphic either to $\mathcal{S}_1$ or $\mathcal{A}_1$ according to whether $\graf$ has an essential vertex or not. In the latter case, by our assumption on cycle components, we have $\graf\cong[0,1]$, and the {\'{S}}wi\k{a}tkowski complex in this case is isomorphic to $\mathcal{A}_1$ with trivial differential. In the former case, the complex is given in degree $0$ by $\mathcal{A}_E$, and an easy calculation shows that the differential imposes relations identifying all $e\in E$ and forcing autocommutativity. Thus, $\mathcal{H}_0(\graf)$ receives a surjection from $\mathcal{S}_1$, which is necessarily an isomorphism, since $\mathcal{H}_0(\graf)$ is nonzero in each weight.
\end{proof}

\section{Asymptotics of functions and symmetric sequences}\label{section:asymptotics}

In this section, we discuss some basic ideas around functions and their asymptotics, which we will apply to the dimension and multiplicity functions of the symmetric sequences appearing on the $E^2$-page of the vertex explosion spectral sequence. The connection between operations on functions and operations on symmetric sequences is given in Section \ref{section:application to symmetric sequences}. Although our discussion is completely elementary, the reader may wish to skip over it on a first reading, returning as necessary to understand the later argument.

\subsection{Functions of integers} In this section, we deal with functions $f:\mathbb{Z}_{\geq0}\to \mathbb{R}_{\geq0}$. We deal with two types of convolution operation on such functions, which are defined by the following formulas:
\begin{align*}
(f \cdot g)(k)&=\sum_{\ell=0}^kf(\ell)g(k-\ell)\\
(f\hspace{.7mm}{\star}\hspace{.7mm}g)(k)&=\sum_{\ell=0}^k\binom{k}{\ell}f(\ell)g(k-\ell)
\end{align*} Each function $f$ has an exponential counterpart $\mathrm{exp}(f)$, defined by $\mathrm{exp}(f)(k)=\frac{f(k)}{k!}$, and these constructions interact according to the equation
\[\mathrm{exp}(f\star g)=\mathrm{exp}(f)\cdot\mathrm{exp}(g).\] Writing $\mathbbm{1}_m$ for the indicator function of $m\in \mathbb{Z}_{\geq0}$, we define two shift operation by 
\begin{align*}
\sigma(f)&=f\cdot \mathbbm{1}_1\\
\hat\sigma(f)&=f\star \mathbbm{1}_1.
\end{align*}
Noting that $\sigma^r(f)(k)=0$ for $k\leq r$, it is sensible to define a summation operation by the formula
\[\Sigma(f)=\sum_{r\geq0} \sigma^r(f).\] 
Our primary interest will be in the growth rates of such functions. To this end, given a function $g$ with co-finite support, we write
\[o(g)=\left\{f\mid \lim_{k\to \infty} \frac{f(k)}{g(k)}=0\right\}.\]

\begin{lemma}\label{lem:shift polynomial}
For any $n\geq0$, we have the equalities 
\begin{align*}
\sigma(k^n)&=k^n+o(k^n)\\
\hat\sigma(k^nk!)&=k^nk!+o(k^nk!)\\
\Sigma(k^n)&=\frac{1}{n+1}k^{n+1}+o(k^{n+1}).
\end{align*}
\end{lemma}
\begin{proof}
By the binomial theorem, the difference $(k-1)^n-k^n$ is a polynomial of degree $n-1$, implying the first claim, which in turn directly implies the second. For the third, we have
\begin{align*}
\lim_{k\to \infty}\frac{\Sigma(k^n)}{k^{n+1}}&=\lim_{k\to \infty}\sum_{\ell=1}^k\left(\frac{\ell}{k}\right)^n\frac{1}{k}= \int_0^1 x^ndx=\frac{1}{n+1},
\end{align*}
implying the claim.
\end{proof}

\begin{corollary}\label{cor:shift little o}
We have $f\in o(k^n)$ if and only if $\sigma(f)\in o(k^n)$. In this case, we also have $\Sigma(f)\in o(k^{n+1})$.
\end{corollary}
\begin{proof}
The first claim is immediate from Lemma \ref{lem:shift polynomial}. For the second claim, our assumption on $f$ grants that, given $\epsilon > 0$, there exists $k_0$ sufficiently large such that $f(k)\leq \epsilon k^n$ for $k\geq k_0$. Hence
    \begin{equation*}
        \lim_{k\to\infty}{\frac{\Sigma(f)(k)}{k^{n+1}}} \leq \lim_{k\to\infty}\left[{\frac{\sum_{\ell = 0}^{k_0-1}{f(\ell)}}{k^{n+1}} + \sum_{\ell = k_0}^k}\epsilon\left(\frac{\ell}{k}\right)^n\frac{1}{k} \right]\leq \frac{\epsilon}{n+1}
    \end{equation*}
   as in the proof of Lemma \ref{lem:shift polynomial}, where we have used that the numerator of the first term is constant. Since $\epsilon$ was arbitrary, the claim follows.
\end{proof}

\begin{corollary}\label{cor:hat shift little o}
If $f\in o(k^nk!)$, then $\hat\sigma(f)\in o(k^nk!)$.
\end{corollary}
\begin{proof}
The premise is equivalent to the containment $\mathrm{exp}(f)\in o(k^n)$, and, since $\mathrm{exp}(\mathbbm{1}_1)=\mathbbm{1}_1$, the conclusion is equivalent to the containment $\sigma(\mathrm{exp}(f))\in o(k^n)$, so the claim follows from Corollary \ref{cor:shift little o}.
\end{proof}

%\begin{lemma}\label{lem:product closure}
%If $f,g\in o(1)$, then $f\star g\in o(1)$.
%\end{lemma}
%\begin{proof}
%Given $f,g\in o(1)$, there exists $k_0$ sufficiently large so that $f(k)=g(k)=0$ for $k\geq k_0$. It then follows directly from the definition that $f\cdot g$ vanishes for $k\geq 2k_0$, so $f\cdot g\in o(1)$.
%\end{proof}
%
%Applying $\mathrm{exp}$, we draw the following conclusion.
%
%\begin{corollary}\label{cor:star product closure}
%If $f,g\in o(k!)$, then $f\star g\in o(k!)$.
%\end{corollary}

\subsection{Functions of partitions}\label{section:functions of partitions}
In this section, we write $\Pi$ for the set of all partitions of non-negative integers, our conventions regarding which are spelled out in Section \ref{section:conventions}. We adopt the common practice of confusing a partition with the corresponding Young diagram; for this and other concepts from the representation theory of symmetric groups, we direct the reader to \cite[Lecture 4]{FultonHarris:RTFC}. 

Our interest lies in non-negative real-valued functions on the set $\Pi$.\footnote{Alternatively, one may view such a function as a choice of class function on $\Sigma_k$ for every $k\geq0$, or as a functional on the space of symmetric polynomials.} The relevant notion of convolution of such functions is given in terms of Littlewood--Richardson coefficients $c^\lambda_{\mu\nu}$ by the formula 
\[(f\star g)(\lambda)=\sum_{\mu,\nu}c_{\mu\nu}^\lambda f(\mu)g(\nu)\] (fortunately, we will need to know almost nothing about these coefficients).
Similarly, we define a shift operation by $\hat\sigma(f)=f\star \mathbbm{1}_{(1)}$, where we use $\mathbbm{1}$ as above to denote an indicator function. This double use of notation will be justified by Proposition \ref{prop:function identities} below.

Given a partition $\lambda$ and $k\geq|\lambda|+\lambda_1$, we write $\lambda[k]=(k-|\lambda|, \lambda_1,\ldots, \lambda_m)$ for the corresponding stabilized or ``padded'' partition of $k$. Given a function $f$, we define $f_\lambda:\mathbb{Z}_{\geq0}\to\mathbb{R}_{\geq0}$ by  
\[
f_\lambda(k)=\begin{cases}
f(\lambda[k])&\quad k\geq|\lambda|+\lambda_1\\
0&\quad \text{otherwise.}
\end{cases}
\] Again, our primary interest is in the asymptotics of these functions. To this end, given a function $g:\mathbb{Z}_{\geq0}\to \mathbb{R}_{\geq0}$ with co-finite support, we write
\[\pi(g)=\left\{f\mid f_\lambda\in o(k^{|\lambda|}g),\, \,\forall\lambda\in\Pi\right\}.\]

In order to state our first result, we introduce the notation $\Pi_{\leq r}(\lambda)$ for the set of partitions that may be obtained from $\lambda$ by removing at most $r$ boxes, no two in the same column (resp. $\Pi_{r}(\lambda)$, exactly $r$ boxes).

\begin{lemma}\label{lem:pieri}
For any function $f$, non-negative integer $r$, and partition $\lambda$, we have
\[(f\star\mathbbm{1}_{(r)})_\lambda=\sum_{\nu\in \Pi_{\leq r}(\lambda)} \sigma^r (f_\nu).\]
\end{lemma}
\begin{proof}
According to the Pieri rule, the Littlewood--Richardson coefficient $c^{\lambda[k]}_{\mu, (r)}$ is nonzero if and only if $\mu\in \Pi_r(\lambda[k])$, in which case $c^{\lambda[k]}_{\mu, (r)}=1$. Defining $\nu$ by dropping the first row of $\mu$, we see that $\nu$ is obtained from $\lambda$ by removing some number of boxes, say $r_0$, and that the remaining $r-r_0$ boxes are removed from the first row of $\lambda[k]$. Thus, we have 
\[\mu=(k-|\lambda|-r+r_0, \nu_1,\ldots, \nu_m)=\nu[k-r]\]
by definition, since $|\lambda|-r_0=|\nu|$. In this way, we have accounted for the terms of the claimed expression with $k-r\geq|\nu|+\nu_1$, but $f_\nu(k-r)=0$ outside of this range.
\end{proof}

In the next result, we write $e=\sum_{r\geq0} \mathbbm{1}_{(r)}$ for the indicator function of the set of trivial partitions, whose topological significance is as the multiplicity function of the twisted algebra $\mathcal{S}_1$.

\begin{corollary}\label{cor:tensor e estimate}
For any function $f$ and partition $\lambda\in\Pi$, we have 
\begin{align*}
(f\star e)_\lambda&=\Sigma(f_\lambda)+\sum_{r>0}\sum_{\mu\in \Pi_{r}(\lambda)}\Sigma(\sigma^r(f_\mu)).
%f_\lambda&=\sum_{\mu\in \Pi_1(\lambda)}\sigma((f\star a)_\mu).
\end{align*}
\end{corollary}
%\begin{proof}
%Both equations follow from Lemma \ref{lem:pieri}, in the second case in view of the identity $(1-\delta_{(1)})\star a=1$.
%\end{proof}

\begin{corollary}\label{cor:shift little pi}
For any $f\in \pi(k^n)$, we have the equalities
\begin{align*}
(f\star e)_\lambda&= \Sigma(f_\lambda)+o(k^{|\lambda|+n})\\
\hat\sigma(f)_\lambda&=f_\lambda+o(k^{|\lambda|+n-1}).
\end{align*} Moreover, we have the containments $f\star e\in \pi(k^{n+1})$, $\hat\sigma(f)\in \pi(k^n)$, and $f-\hat\sigma(f)\in \pi(k^{n-1})$. 
%and $f\star a\in \pi(k^{n+1})$.
\end{corollary}
\begin{proof}
The first equality is immediate from Corollaries \ref{cor:shift little o} and \ref{cor:tensor e estimate} and the fact that $|\mu|<|\lambda|$ for $\mu\in \Pi_r(\lambda)$ with $r>0$, and the second follows in the same manner from Lemma \ref{lem:pieri}. The claimed containments follow in view of Corollary \ref{cor:shift little o}.
\end{proof}

\subsection{Application to symmetric sequences}\label{section:application to symmetric sequences} A symmetric sequence $\mathcal{X}$ determines a function $\dim \mathcal{X}$ on $\mathbb{Z}_{\geq0}$ and a function $\mul\,\mathcal{X}$ on $\Pi$ by the formulas
\begin{align*}
\dim \mathcal{X}(k)&=\dim\mathcal{X}_k\\
\mul\,\mathcal{X}(\lambda)&=\dim\Hom_{\Sigma_{|\lambda|}}\left(V_\lambda, \mathcal{X}_{|\lambda|}\right),
\end{align*} where $V_\lambda$ is the Specht module indexed by $\lambda$. We will only consider the latter function in characteristic zero, in which case the Specht modules are precisely (up to isomorphism) the irreducible representations, so that $\mul\,\mathcal{X}(\lambda)$ is simply the multiplicity of $V_\lambda$ as a summand of $\mathcal{X}_{|\lambda|}$.

The connection between the algebra of symmetric sequences and the functional operations introduced above is recorded in the following result:

\begin{proposition}\label{prop:function identities}
For symmetric sequences $\mathcal{X}$ and $\mathcal{Y}$, we have the following equalities:
\begin{align*}
\dim\mathcal{X}\otimes\mathcal{Y}&=\dim\mathcal{X}\star\dim\mathcal{Y}\\
\mul\,\mathcal{X}\otimes\mathcal{Y}&=\mul\,\mathcal{X}\star\mul\,\mathcal{Y}\\
\hat\sigma(\dim\mathcal{X})&=\dim \mathcal{X}\otimes \mathbb{F}(1)\\
\hat\sigma(\mul\,\mathcal{X})&=\mul\, \mathcal{X}\otimes \mathbb{F}(1).
\end{align*}
\end{proposition}

Abusively, we will use the notation $o(f)$ to refer to the full subcategory of symmetric sequences $\mathcal{X}$ such that $\dim\mathcal{X}\in o(f)$ (resp. $\mul$, $\pi(f)$). Concerning these subcategories, we record the following simple observation.

\begin{proposition}\label{prop:closure properties}
The subcategories $o(f)$ and $\pi(f)$ are Serre subcategories, and the subcategories $o(k^nk!)$ and $\pi(k^n)$ are closed under the formation of shifts.
\end{proposition}
\begin{proof}
The first claim is obvious, and the second claim follows from Proposition \ref{prop:function identities} and Corollaries \ref{cor:hat shift little o} and \ref{cor:shift little pi}. \end{proof}

\section{Toward the main results}\label{section:toward the main results}

As indicated in the introduction, the argument for Theorem \ref{thm:conceptual asymptotics}, and its computational consequences in Theorems \ref{thm:asymptotics} and \ref{thm:rep asymptotics}, whose statements the reader may wish to re-read before continuing, is essentially an elaborate induction. The purpose of this section is threefold: first, in Theorem \ref{thm:H0}, we prove the base case of the induction, which is to say the case of degree $i=0$; second, we state the main reduction to be achieved by the inductive argument, namely Theorem \ref{thm:upper bound}; third, we combine the two to deduce the main results. The more technical inductive argument is postponed to Section \ref{section:technical estimate} below.

\subsection{Base case} The basic calculation underpinning our approach is the following, which, in view of Proposition \ref{prop:H0 description}, amounts to an asymptotic count of the connected components of the configuration spaces of a graph.

\begin{theorem}\label{thm:H0}
The symmetric sequence $\mathcal{S}_m\otimes \mathcal{A}_n$ lies in $o(k^nk!)$; more precisely, we have 
\[
\dim \mathcal{S}_m\otimes \mathcal{A}_n=\begin{cases}
\frac{e^m}{(n-1)!}k^{n-1}k!+o(k^{n-1}k!)&\quad n>0\\
m^k&\quad n=0.
\end{cases}
\]
In characteristic zero, this symmetric sequence also lies in $\pi(k^{m+n})$; more precisely, we have 
\[
\mul_\lambda\, \mathcal{S}_m\otimes\mathcal{A}_n=\begin{cases}
\frac{\dim V_\lambda}{(|\lambda|+m+n-1)!}\binom{|\lambda|+n-1}{n-1}k^{|\lambda|+m+n-1}+o(k^{|\lambda|+m+n-1})&\quad n>0\\
\frac{c_\lambda}{(m-1)!} k^{m-1}+o(k^{m-1})&\quad n=0,
\end{cases}
\] where $c_\lambda$ is a non-negative constant such that $c_\varnothing=1$.
\end{theorem}

\begin{remark}
In fact, one can show that $c_\lambda=s_\lambda(1,\ldots, 1)$, where $s_\lambda$ is the Schur polynomial associated to $\lambda$ (see \cite[Lecture 6]{FultonHarris:RTFC}. As we do not require this identity, we content ourselves with a sketch of proof. First, by repeated application of Pieri's rule, one shows that $\mul\,\mathcal{S}_m(\mu)=e^{\star m}(\mu)$ is equal to the number of semi-standard Young tableaux of shape $\mu$ with entries drawn from $\{1,\ldots, m\}$, which is known to equal $s_\mu(1,\ldots, 1)$. This number may be explicitly calculated using Weyl's character formula (see \cite[Thm. 6.3]{FultonHarris:RTFC}). For $\mu=\lambda[k]$, an argument along the lines of Lemma \ref{lem:hook length} below shows it to be a polynomial in $k$ with the claimed degree and leading coefficient.
\end{remark}

\begin{remark}
In somewhat less formal terms, Theorem \ref{thm:H0} asserts that, asymptotically, every copy of $V_{\lambda[k]}$ in $\mathcal{S}_m\otimes\mathcal{A}_n$ arises from a copy of $V_{\lambda[\ell]}$ in $\mathcal{A}_n$ by further padding the first row of $\lambda[\ell]$.
\end{remark}

We require two simple lemmas for the proof, the first an easy fact about binomial coefficients.

\begin{lemma}\label{lem:binomial coefficient}
For any $n\geq0$, we have $\binom{k+n}{n}=\frac{1}{n!}k^n+o(k^{n})$
\end{lemma}
\begin{proof}
By definition, we have $n!\binom{k+n}{n}=\bigsqcap_{m=1}^n (k+m)$, a monic polynomial of exact degree $n$.
\end{proof}

The second input concerns the asymptotics of Specht modules.

\begin{lemma}\label{lem:hook length}
For any partition $\lambda$, we have $\dim V_{\lambda[k]}=\frac{\dim V_\lambda}{|\lambda|!}k^{|\lambda|}+o(k^{|\lambda|})$.
\end{lemma}

We defer the proof of this estimate to the end of the section to avoid distraction.

\begin{proof}[Proof of Theorem \ref{thm:H0}]
For the first equation, it is direct from the definitions that 
\begin{align*}
(\mathcal{S}_m\otimes \mathcal{A}_n)_k
%&\cong \bigoplus_{k_1+\cdots+k_r=k}\mathrm{Ind}^{\Sigma_k}_{\Sigma_{k_1}\times\cdots\times\Sigma_{k_r}}\bigotimes_{a=1}^rH_0(F_{k_a}(\graf_a))\\
%&\cong\bigoplus_{k_1+\cdots+k_r=k}\mathrm{Ind}^{\Sigma_k}_{\Sigma_{k_1}\times\cdots\times\Sigma_{k_r}}\bigotimes_{a=s+1}^r\Sigma_{k_a}\\
&\cong\bigoplus_{k_1+\cdots+k_{m+n}=k}\mathbb{F}\left\langle\frac{\Sigma_k}{\Sigma_{k_{1}}\times\cdots\times\Sigma_{k_m}}\right\rangle
%&=\bigoplus_{\ell\leq k}\left[\bigoplus_{\ell_1+\cdots+\ell_s=\ell}\frac{\Sigma_k}{\Sigma_{\ell_{1}}\times\cdots\times\Sigma_{\ell_s}}\right]^{\binom{k-\ell+r-s-1}{r-s-1}}.
\end{align*} as $\Sigma_k$-representations. Setting $\ell=k_1+\cdots+k_m$, and noting that the number of weak compositions of $k-\ell$ with $n$ terms is $\binom{k-\ell+n-1}{n-1}$, we conclude that
\begin{align*}
\dim\mathcal{S}_m\otimes \mathcal{A}_n&=\sum_{\ell\leq k}\binom{k-\ell+n-1}{n-1}\sum_{k_1+\cdots+k_m=\ell}\frac{k!}{k_1!\cdots k_m!}\\
&=k!\sum_{\ell\leq k}\binom{k-\ell+n-1}{n-1}\frac{m^\ell}{\ell!}.
\end{align*}
The calculation for $n=0$ follows, as in this case only the term with $\ell=k$ is nonzero, and Stirling's approximation implies that $\dim \mathcal{S}_m\in o(k!)$. To determine the asymptotic behavior for $n>1$, consider the sequence of functions $f_k:\mathbb{Z}_{\geq0}\to\mathbb{R}$ given by the formula
\[
f_k(\ell)=
\begin{cases}
\frac{(n-1)!}{k^{n-1}}\binom{k-\ell+n-1}{n-1}\frac{m^\ell}{\ell!}&\quad\text{if } \ell\leq k\\
0&\quad\text{otherwise.}
\end{cases}
\] For $k\geq n-1$, a fixed quantity, we have $f_k(\ell)\leq \frac{2m^\ell}{\ell!}$, so the dominated convergence theorem gives
\begin{align*}
\lim_{k\to\infty}\frac{(n-1)!}{k^{n-1}k!}\dim(\mathcal{S}_m\otimes \mathcal{A}_n)_k=\lim_{k\to\infty}\sum_{\ell=0}^\infty f_k(\ell)
=\sum_{\ell=0}^\infty \lim_{k\to \infty}f_k(\ell)
=\sum_{\ell=0}^\infty \frac{m^\ell}{\ell!}
=e^m,
\end{align*}
establishing the first equality. 

For the second equation, we fix $n>0$ and proceed by induction on $m$. For the base case $m=0$, the calculation above shows that $(\mathcal{A}_n)_k$ is a direct sum of $\binom{k+n-1}{n-1}$ copies of the regular representation. Thus, by Lemmas \ref{lem:binomial coefficient} and \ref{lem:hook length}, we have
\[\mul_\lambda\, \mathcal{A}_n=\frac{\dim V_\lambda}{|\lambda|!(n-1)!}k^{|\lambda|+n-1}+o(k^{|\lambda|+n-1}),\] which coincides with the claimed expression in the case of interest. For the induction step, we calculate using Proposition \ref{prop:function identities} that 
\begin{align*}
\mul_\lambda\,\mathcal{S}_m\otimes \mathcal{A}_n&=\left(\mul\,\left(\mathcal{S}_{m-1}\otimes \mathcal{A}_n\right)\otimes\mathcal{S}_1\right)_\lambda\\
&=(\mul\, \mathcal{S}_{m-1}\otimes \mathcal{A}_n\star \mul\,\mathcal{S}_1)_\lambda\\
&=(\mul\, \mathcal{S}_{m-1}\otimes \mathcal{A}_n\star e)_\lambda\\
&=\Sigma\left(\mul_\lambda\,\mathcal{S}_{m-1}\otimes \mathcal{A}_n\right)+o(k^{|\lambda|+m+n-1}))\\
&=\frac{\dim V_\lambda}{(|\lambda|+m+n-1)!}\binom{|\lambda|+n-1}{n-1}k^{|\lambda|+m+n-1}+o(k^{|\lambda|+m+n-1}),
\end{align*} where we have used Corollary \ref{cor:shift little pi} in the fourth line and Lemma \ref{lem:shift polynomial} in the fifth.

Finally, the claim for $n=0$ follows by an easy induction using Lemmas \ref{lem:shift polynomial} and Corollaries \ref{cor:shift little o} and \ref{cor:shift little pi}.
\end{proof}

%\begin{corollary}\label{cor:shift H0}
%With the assumptions of Theorem \ref{thm:H0}, the function $\dim \mathcal{S}_m\otimes\mathcal{A}_n$ is fixed by $\hat\sigma$ modulo $o(k^{n-1}k!)$ (resp.  $\mul$, $\pi(k^{m+n-1})$).
%\end{corollary}
%\begin{proof}
%The first claim follows by applying Lemma \ref{lem:shift polynomial} and Corollary \ref{cor:shift little o} to the calculation of Theorem \ref{thm:H0}. For the second, Theorem \ref{thm:H0} implies that $\mul_\nu\,\mathcal{S}_m\otimes \mathcal{A}_n\in o(k^{|\lambda|+m+n-1})$ whenever $|\nu|<|\lambda|$, so the claim follows from Lemma \ref{lem:pieri} and Corollary \ref{cor:shift little pi}.
%\end{proof}

\begin{proof}[Proof of Lemma \ref{lem:hook length}]
The hook length formula expresses the quantity in question as the ratio of $k!$ to the product of the lengths of the ``hooks'' in (the Young diagram of) $\lambda[k]$. We may formulate this expression as
\[\dim V_{\lambda[k]}=\frac{1}{C_1}\frac{k!}{C_2},\] where $C_1$ is the product of the lengths of the hooks in $\lambda[k]$ that do not involve the first row, and $C_2$ is the product of the rest. Since the first factor in this expression is $\frac{\dim V_\lambda}{|\lambda|!}$, it suffices to estimate the second. Writing $(\mu_1,\ldots, \mu_{\lambda_1})$ for the conjugate partition of $\lambda$, and noting that the first row of $\lambda[k]$ has length $k-|\lambda|$, we have
\[C_2=\bigsqcap_{i=1}^{\lambda_1} (k-|\lambda|-i+1+m_i)\bigsqcap_{j=1}^{k-|\lambda|-\lambda_1}j,\] where the second product represents the hooks that lie entirely in the first row of $\lambda[k]$. Thus, we have 
\[\frac{k!}{C_2}=\frac{\displaystyle\bigsqcap_{j=1}^{|\lambda|+\lambda_1}(k-j+1)}{\displaystyle\bigsqcap_{i=1}^{\lambda_1} (k-|\lambda|-i+1+m_i)}=k^{|\lambda|}+o(k^{|\lambda|}).\]
\end{proof}

\subsection{Main reduction}\label{section:main reduction}

Throughout this section, we fix a graph $\graf$ with no cycle component and a set $B$ of bivalent vertices such that $\graf_B$ has at most one essential vertex in each component---see Figure \ref{fig:explosion} for an example. We assume that such a set $B$ exists, which is always achievable after subdivision. Writing $W$ for the set of essential vertices of $\graf$, which we regard as a subset of $\pi_0(\graf_B)$, we have \[\graf_B= \bigsqcup_{w\in W}\stargraph{d(w)}\sqcup \bigsqcup_{ \pi_0(\graf_B)\setminus W}[0,1].\]

The principal reduction in the proof of our main results is the following theorem, which asserts that $\mathcal{H}_i(\graf)$ is dominated asymptotically by homology classes arising from an $i$-fold disjoint union of star graphs, corresponding to the $p=0$ column of the vertex explosion spectral sequence.\footnote{A simple example of a class lying in higher filtration is the loop class, i.e., the class of a single particle traversing a cycle. See \cite[\S 4.2]{ChettihLuetgehetmann:HCSTL} for a much subtler example.}

\begin{theorem}\label{thm:upper bound}
Let $\graf$ be a connected graph not homeomorphic to the circle. The canonical edge map $\mathrm{Tor}_0^{\mathcal{A}_B}(\mathcal{H}_i(\graf_B))\to \mathcal{H}_i(\graf)$ is an isomorphism modulo $o(k^{\Lambda^i_\graf-1}k!)$, and
{
\[\dim\mathrm{Tor}^{\mathcal{A}_B}_0(\mathcal{H}_i(\graf_B))=\sum_{|S|=i}\left[\bigsqcap_{w\in S}\left(\left( d(w)-1\right)\hat\sigma-1\right)\right]\dim\mathcal{H}_0(\graf_S)+o(k^{\Lambda^{i}_\graf-1}k!),\]}where the sum is taken over sets of essential vertices $S$ of cardinality $i$. In characteristic zero, the same map is an isomorphism modulo $\pi(k^{\Delta^i_\graf-1})$, and
{
\[\mul\,\mathrm{Tor}^{\mathcal{A}_B}_0(\mathcal{H}_i(\graf_B))=\sum_{|S|=i}\left[\bigsqcap_{w\in S}\left(\left(d(w)-1\right)\hat \sigma-1\right)\right]\mul\,\mathcal{H}_0(\graf_S)+\pi(k^{\Delta^{i}_\graf-1}).\]}
\end{theorem}

This result, together with Lemma \ref{lem:shift polynomial} (or Corollary \ref{cor:shift little pi} in characteristic zero), immediately implies Theorem \ref{thm:conceptual asymptotics}.

\begin{proof}[Proof of Theorems \ref{thm:asymptotics} and \ref{thm:rep asymptotics}] Using Theorem \ref{thm:conceptual asymptotics}, Proposition \ref{prop:H0 description}, and Theorem \ref{thm:H0}, we calculate that 
\begin{align*}
\dim \mathcal{H}_i(\graf)&\approx\sum_{|S|=i}\left[\bigsqcap_{w\in S}\left(d(w)-2\right)\right]\dim\mathcal{H}_0(\graf_S)\\
&=\sum_{|S|=i}\left[\bigsqcap_{w\in S}\left(d(w)-2\right)\right]\dim\mathcal{S}_{\mathcal{E}^S_\graf}\otimes \mathcal{A}_{\Lambda^S_\graf}\\
&\approx\sum_{|S|=i,\, \Lambda^S_\graf>0}\frac{e^{\mathcal{E}^S_\graf}}{(\Lambda^S_\graf-1)!}\left[\bigsqcap_{w\in S} (d(w)-2)\right]k^{\Lambda^S_\graf-1}k!\\
&\approx\left[\sum_{\Lambda^S_\graf=\Lambda^i_\graf}\frac{e^{\mathcal{E}^S_\graf}}{(\Lambda^i_\graf-1)!}\bigsqcap_{w\in S} (d(w)-2)\right]k^{\Lambda^i_\graf-1}k!
\end{align*}
In characteristic zero, the same reasoning establishes that
\[\mul_\lambda\, \mathcal{H}_i(\graf)\approx \sum_{|S|=i}\left[\bigsqcap_{w\in S}\left(d(w)-2\right)\right]\mul_\lambda\,\mathcal{S}_{\mathcal{E}^S_\graf}\otimes \mathcal{A}_{\Lambda^S_\graf}.\]Specializing to $\lambda$ and consulting Theorem \ref{thm:H0}, we now confront three cases. If $\lambda=\varnothing$, the claim follows as before. Otherwise, we note that the summand indexed by $S$ exhibits asymptotic growth of exact degree $|\lambda|+\Delta^S_\graf-1$ or degree at most $\Delta^S_\graf-1=\mathcal{E}^S_\graf-1$ according to whether $\Lambda^S_\graf$ vanishes. If $\Lambda_\graf^S>0$ whenever $\Delta_\graf^S=\Delta_\graf^i$, or if $|\lambda|\geq \mathcal{E}^i_\graf$, then the second type of term does not contribute asymptotically, and the conclusion again follows. 
\end{proof}

\section{Technical estimates}\label{section:technical estimate}

In this section, we complete the argument by establishing Theorem \ref{thm:upper bound}. The key technical result here is Theorem \ref{thm:technical bound} below, which is an estimate of the growth of certain Tor terms. The terms of greatest interest for us are those appearing on the $E^2$-page of the vertex explosion spectral sequence; however, in estimating these terms, it will be convenient to pass through estimates of terms arising from certain auxiliary modules suggested by Proposition \ref{prop:star explosion}. Throughout, we maintain the notation of Section \ref{section:main reduction}.

\subsection{Auxiliary modules} We begin by defining a collection of modules which in some sense interpolate among various modules of direct interest. These modules are constructed by tensoring together assignments made to the various connected components of $\graf_B$.

\begin{construction}
Given a set partition $W=S\sqcup T\sqcup U\sqcup V$, we define a symmetric sequence $\mathcal{H}_{S,T,U,V}=\mathcal{H}_{S,T,U,V}(\graf_B)$ by
\[\mathcal{H}_{S,T,U,V}=\bigotimes_{w\in S} \mathcal{H}_1(\stargraph{d(w)})\otimes \bigotimes _{w\in T}\mathcal{A}_{d(w)}\otimes\bigotimes_{w\in U}\mathcal{N}_{d(w)}\otimes \bigotimes_{w\in V} \mathcal{H}_0(\stargraph{d(w)})\otimes\mathcal{A}_{\pi_0(\graf_B)\setminus W},\] where $\mathcal{N}_n=\ker(\mathcal{A}_n\to \mathcal{H}_0(\stargraph{n}))$. We regard this symmetric sequence as an $\mathcal{A}_B$-module via the homomorphism $\mathcal{A}_B\to \mathcal{A}_{L(\graf_B)}$ sending a bivalent vertex to the difference of the associated leaves.
\end{construction}

The point of this construction is that different choices of subscript are interrelated by a pair of exact sequences, permitting interpolation among various extremal choices.

\begin{lemma}\label{lem:exact sequences}
Given a set partition $W=S\sqcup T\sqcup U\sqcup V$, there are canonical exact sequences of the following form:
\[
0\to\mathcal{H}_{S,T,U,V}\to \mathcal{H}_{S\setminus\{w\}, T\cup\{w\}, U, V}\otimes\delta_{d(w)}(1)\to \mathcal{H}_{S\setminus \{w\},T,U\cup\{w\},V}\to 0
\]
\[
0\to \mathcal{H}_{S,T,U,V}\to \mathcal{H}_{S, T\cup\{w\}, U\setminus\{w\}, V}\to \mathcal{H}_{S,T,U\setminus \{w\},V\cup\{w\}}\to  0
\]
\end{lemma}
\begin{proof}
The claim is immediate from Proposition \ref{prop:star explosion} and flatness.
%
%which supplies the exact sequences
%\[0\to \mathcal{H}_1(\stargraph{d(w)})\to \mathcal{A}_{d(w)}\otimes\delta_{d(w)}(1)\to \mathcal{N}_{d(w)}\to 0\]
%\[0\to \mathcal{N}_{d(w)}\to \mathcal{A}_{d(w)}\to \mathcal{H}_0(\stargraph{d(w)})\to 0\]
\end{proof}

We close this section with three simple but important observations regarding these modules, in which one can already begin to discern the shape of the coming argument. First, after applying Tor, we have the following method for reducing the second subscript.

\begin{lemma}\label{lem:shift explosion}
We have $\mathrm{Tor}^{\mathcal{A}_B}_p(\mathcal{H}_{S,T,U,V}(\graf_B))\cong \mathrm{Tor}_p^{\mathcal{A}_{B_w}}\left(\mathcal{H}_{S,T\setminus\{w\}, U,V}\left(\graf_{B_w\cup \{w\}}\right)\right)$, where $B_w\subseteq B$ is the subset of bivalent vertices not adjacent to $w$.
\end{lemma}
\begin{proof}
In view of the isomorphism $\mathcal{A}_B\cong \mathcal{A}_{B_w}\otimes \mathcal{A}_{B_w^c}$, the claim follows from the definitions by repeated application of the isomorphism
\[(\mathcal{M}\otimes\mathcal{A}_1)\otimes_{\mathcal{A}_1}\mathbb{F}\cong\mathcal{M}\otimes_{\mathcal{A}_1}\mathcal{A}_1\cong\mathcal{M},\] where $\mathcal{M}$ is an arbitrary $\mathcal{A}_1$-module and $\mathcal{A}_1$ acts on $\mathcal{M}\otimes\mathcal{A}_1$ by the difference of the generators.
\end{proof}

Second, at one extreme of subscript, the auxiliary modules form direct sum decompositions of the homologies involved in the $E^2$-page of the vertex explosion spectral sequence.

\begin{lemma}\label{lem:decomposition}
For any $i\geq0$, there is a canonical isomorphism
\[\mathcal{H}_i(\graf_B)\cong\bigoplus_{|S|=i} \mathcal{H}_{S,\varnothing,\varnothing, W\setminus S}.\]
\end{lemma}
\begin{proof}
The claim is immediate from Example \ref{example:disjoint union} and Proposition \ref{prop:H0 description}.
\end{proof}

Third, at another extreme of subscript, we have the following geometric interpretation after applying $\mathrm{Tor}_0$.

\begin{lemma}\label{lem:H0}
For any $S\subseteq W$, there is a canonical isomorphism
\[\mathrm{Tor}^{\mathcal{A}_B}_0(\mathcal{H}_{\varnothing,S,\varnothing,W\setminus S})\cong\mathcal{H}_0(\graf_S).\]
\end{lemma}
\begin{proof}
The claim is immediate from Lemma \ref{lem:shift explosion} and consideration of the entry in bidegree $(0,0)$ of the vertex explosion spectral sequence.
\end{proof}

\subsection{Main estimate} With our auxiliary modules at hand, we can now state our desired technical bound.

\begin{theorem}\label{thm:technical bound}
Fix $\graf$ and $B$ as above. For any set partition $W=S\sqcup T\sqcup U\sqcup V$ and $p\geq0$,  we have $\mathrm{Tor}_p^{\mathcal{A}_B}(\mathcal{H}_{S,T,U,V})\in o(k^{\Lambda^{S\cup U}_{\graf_T}}k!)$. In characteristic zero, we have $\mathrm{Tor}_p^{\mathcal{A}_B}(\mathcal{H}_{S,T,U,V})\in \pi(k^{\Delta^{S\cup U}_{\graf_T}})$.
\end{theorem}

In our later applications of this result, we will use the fact that $\Lambda^{S\cup U}_{\graf_T}\leq \Lambda^{|S|+|U|}_{\graf_T}\leq \Lambda^{|W\setminus V|}_\graf$ by Lemma \ref{lem:explosion inequality}, and similarly for $\Delta$.

The proof will proceed by reduction to the following special case, which we take up in the final section of the paper.

\begin{lemma}\label{lem:base case}
For every $p\geq0$, we have $\mathrm{Tor}^{\mathcal{A}_B}_p(\mathcal{H}_{\varnothing,\varnothing,\varnothing,W})\in o(k^{\Lambda_\graf^0}k!)$. In characteristic zero, we have $\mathrm{Tor}^{\mathcal{A}_B}_p(\mathcal{H}_{\varnothing,\varnothing,\varnothing,W})\in \pi(k^{\Delta_\graf^0})$. 
\end{lemma}

\begin{proof}[Proof of Theorem \ref{thm:technical bound}]
We give the proof of the first claim, the second being essentially identical. The argument will proceed through a series of reductions. We begin by considering the long exact sequence
{\tiny
\[\cdots\to  \mathrm{Tor}^{\mathcal{A}_B}_{p+1}(\mathcal{H}_{S\setminus \{w\},T,U\cup\{w\},V})\to \mathrm{Tor}^{\mathcal{A}_B}_p(\mathcal{H}_{S,T, U, V})\to \mathrm{Tor}^{\mathcal{A}_B}_p(\mathcal{H}_{S\setminus\{w\},T\cup\{w\}, U,V})\otimes 
\delta_{d(w)}(1)\to \cdots \]}obtained from the first short exact sequence of Lemma \ref{lem:exact sequences}. By Proposition \ref{prop:closure properties} it suffices to show that the outer two terms lie in $o(k^{\Lambda_{\graf_T}^{S\cup U}}k!)$. For the first term, the containment follows by induction on $|S|$; for the third term, we use induction on $|S|$, Corollary \ref{cor:hat shift little o}, and Lemma \ref{lem:explosion inequality}. On the other hand, we have the long exact sequence
{\tiny
\[\cdots\to \mathrm{Tor}^{\mathcal{A}_B}_{p+1}(\mathcal{H}_{S,T, U\setminus \{w\},V\cup\{w\}})\to  \mathrm{Tor}^{\mathcal{A}_B}_p(\mathcal{H}_{S,T,U,V})\to \mathrm{Tor}^{\mathcal{A}_B}_p(\mathcal{H}_{S,T\cup\{w\}, U\setminus\{w\},V})\to \cdots\]}obtained from the second short exact sequence of Lemma \ref{lem:exact sequences}, and we may apply essentially the same argument by induction instead on $|U|$. These observations reduce the general case to the case $S=U=\varnothing$, and we may further assume that $T=\varnothing$, since
\[\mathrm{Tor}^{\mathcal{A}_B}_p\left(\mathcal{H}_{S,T,U,V}\left(\graf_B\right)\right)\cong \mathrm{Tor}_p^{\mathcal{A}_{B_T}}\left(\mathcal{H}_{S,\varnothing, U,V}\left(\left(\graf_T\right)_{B_T}\right)\right)\] through repeated use of Lemma \ref{lem:shift explosion}. Appealing to Lemma \ref{lem:base case} completes the proof.
\end{proof}

Now, let us see how this result implies the main reduction above. As a preliminary, we have the following result, which is the final missing ingredient.

\begin{lemma}\label{lem:tor ses}
For any set partition $W=S\sqcup T\sqcup U\sqcup V$, $w\in S$, and $p\geq0$, the sequence
{\tiny\[0\to \mathrm{Tor}_p^{\mathcal{A}_B}(\mathcal{H}_{S,T,U,V})\to \mathrm{Tor}_p^{\mathcal{A}_B}(\mathcal{H}_{S\setminus\{w\},T\cup\{w\},U,V})\otimes\delta_{d(w)}\to \mathrm{Tor}_p^{\mathcal{A}_B}(\mathcal{H}_{S\setminus\{w\},T\cup\{w\},U,V})\to 0 \]}is exact modulo $o(k^{\Lambda^{|S|+|U|}_{\graf_T}-1}k!)$. In characteristic zero, the same claims hold modulo $\pi(k^{\Delta^{|S|+|U|}_{\graf_T}-1})$.
\end{lemma}
\begin{proof}
In view of Lemma \ref{lem:exact sequences}, the claim follows from the two claims that the following natural maps are an isomorphism and an epimorphism, respectively, modulo the appropriate class of symmetric sequences: \begin{align*}
\mathrm{Tor}_p^{\mathcal{A}_B}(\mathcal{H}_{S\setminus\{w\},T,U\cup\{w\},V})&\to \mathrm{Tor}_p^{\mathcal{A}_B}(\mathcal{H}_{S\setminus\{w\},T\cup\{w\},U,V})\\
\mathrm{Tor}_p^{\mathcal{A}_B}(\mathcal{H}_{S\setminus\{w\}, T\cup\{w\}, U, V})\otimes\delta_{d(w)}(1)&\to \mathrm{Tor}_p^{\mathcal{A}_B}(\mathcal{H}_{S\setminus\{w\},T,U\cup\{w\},V}).
\end{align*}
For the first claim, as above, it suffices to note that $\mathrm{Tor}_a(\mathcal{H}_{S\setminus\{w\},T, U, V\cup\{w\}})$ lies in the appropriate collection of symmetric sequences for $a\in \{p,p+1\}$ by Theorem \ref{thm:technical bound}. 

For the second claim, we may assume that $\Lambda^{S\cup U}_{\graf_T}=\Lambda^{|S|+|U|}_{\graf_T}$ by Theorem \ref{thm:technical bound} (resp. $\Delta$), since otherwise the target of the map in question already lies in the appropriate collection of symmetric sequences. Next, we observe that the domain of the map in question contains $\mathrm{Tor}_p^{\mathcal{A}_B}(\mathcal{H}_{S\setminus\{w\}, T\cup\{w\}, U, V})\otimes \mathbb{F}(1)$ as a canonical summand for each difference of half-edges at $w$, and it suffices to bound the growth of the cokernel of the map $\varphi$ obtained by restricting to any one of these summands. Given such a half-edge difference, writing $\graf'$ for the result of disconnecting the remaining half-edges at $w$, and noting that $w$ is bivalent in $\graf'$, we have
\begin{align*}
\coker(\varphi)&\cong\mathrm{Tor}_0^{\mathcal{A}_1}\left(\mathrm{Tor}_p^{\mathcal{A}_B}\left(\mathcal{H}_{S\setminus\{w\}, T\cup\{w\}, U, V}\left(\graf_B\right)\right)\right)\\
&\cong \mathrm{Tor}_0^{\mathcal{A}_1}\left(\mathrm{Tor}_p^{\mathcal{A}_{B_{w}}}\left(\mathcal{H}_{S\setminus\{w\}, T, U, V}\left(\left(\graf_w\right)_{B_w}\right)\right)\right)\\
&\cong \mathrm{Tor}_0^{\mathcal{A}_1}\left(\mathrm{Tor}_p^{\mathcal{A}_{B_{w}}}\left(\mathcal{H}_{S\setminus\{w\}, T, U, V}\left(\graf'_{B_w\cup\{w\}}\right)\right)\right)\\
&\subseteq\mathrm{Tor}_p^{\mathcal{A}_{B_{w}\cup\{w\}}}\left(\mathcal{H}_{S\setminus\{w\}, T, U, V}\left(\graf'_{B_w\cup\{w\}}\right)\right)
\end{align*} where we have used Corollary \ref{cor:koszul tor} in the first line, Lemma \ref{lem:shift explosion} in the second, and Corollary \ref{cor:iterated tor} in the fourth. By Theorem \ref{thm:technical bound}, this last object lies in $o(k^{\Lambda^{S\cup U\setminus\{w\}}_{\graf'_T}}k!)$ or $\pi(k^{\Delta^{S\cup U\setminus\{w\}}_{\graf'_T}})$, respectively. 

In the first instance, in view of our assumption that $\Lambda^{S\cup U}_{\graf_T}=\Lambda^{|S|+|U|}_{\graf_T}$, some half-edge at $w$ belongs to an interval component of $(\graf_T)_{S\cup U}$ and another half-edge belongs to some other component. For this pair of half-edges, we have $\Lambda^{S\cup U\setminus\{w\}}_{\graf'_T}<\Lambda^{S\cup U}_{\graf_T}$, implyhing the claim. In the characteristic zero case, we need only choose half-edges belonging to distinct components of $(\graf_T)_{S\cup U}$, which exist by our assumption that $\Delta^{S\cup U}_{\graf_T}=\Delta^{|S|+|U|}_{\graf_T}$
\end{proof}

%\begin{corollary}\label{cor:replacement}
%For any set partition $W=S\sqcup T\sqcup U\sqcup V$, and $p\geq0$, the canonical map \[\mathrm{Tor}_p^{\mathcal{A}_B}(\mathcal{H}_{S,T,U,V})\to \mathrm{Tor}_p^{\mathcal{A}_B}(\mathcal{H}_{S,T\cup\{w\},U\setminus\{w\},V})\] is an isomorphism modulo $o(k^{\Lambda^{|S|+|U|}_{\graf_T}-1}k!)$, and the canonical map
%\[
%\mathrm{Tor}_p^{\mathcal{A}_B}(\mathcal{H}_{S, T\cup\{w\}, U\setminus\{w\}, V})\otimes\delta_{d(w)}(1)\to \mathrm{Tor}_p^{\mathcal{A}_B}(\mathcal{H}_{S,T,U,V})
%\] is surjective modulo $o(k^{\Lambda^{|S|+|U|}_{\graf_T}-1}k!)$. In characteristic zero, the same claims hold modulo $\pi(k^{\Delta^{|S|+|U|}_{\graf_T}-1})$.
%\end{corollary}
%\begin{proof}
%. For the second claim, consider the composite
%\[\mathcal{H}_{S, T\cup\{w\}, U\setminus\{w\}, V}\otimes\delta_{d(w)}(1)\to \mathcal{H}_{S,T,U,V}\to\mathcal{H}_{S,T\cup\{w\},U\setminus\{w\},V}.\] By the first claim, it suffices to show that the cokernel of the induced map after applying $\mathrm{Tor}_p^{\mathcal{A}_B}$ lies in the appropriate collection of symmetric sequences, but this cokernel is a subobject of $\mathrm{Tor}_p^{\mathcal{A}_B}(\mathcal{H}_{S,T,U\setminus\{w\},V\cup\{w\}})$, and the conclusion follows as before.
%\end{proof}

\begin{corollary}\label{cor:move a vertex}
For any set partition $W=S\sqcup T\sqcup U\sqcup V$ and $w\in S$, we have  the equality
{\small
\[\dim\mathrm{Tor}_0^{\mathcal{A}_B}(\mathcal{H}_{S,T,U,V})\approx\left(\left(d(w)-1\right)\hat \sigma-1\right)\dim\mathrm{Tor}_0^{\mathcal{A}_B}(\mathcal{H}_{S\setminus\{w\},T\cup\{w\},U,V}),\]
}and, in characteristic zero, we have the equality
{
\[\mul\,\mathrm{Tor}_0^{\mathcal{A}_B}(\mathcal{H}_{S,T,U,V})\approx\left(\left(d(w)-1\right)\hat \sigma-1\right)\mul\,\mathrm{Tor}_0^{\mathcal{A}_B}(\mathcal{H}_{S\setminus\{w\},T\cup\{w\},U,V}),\]}where the symbol $\approx$ indicates equality modulo $o(k^{\Lambda^{|S|+ |U|}_{\graf_T}-1} k!)$ and equality modulo $\pi(k^{\Delta^{|S|+ |U|}_{\graf_T}-1})$, respectively.
\end{corollary}
%\begin{proof}
%Appealing to the long exact sequences arising from the short exact sequences of Lemma \ref{lem:exact sequences}, we have
%{\small
%\begin{align*}\dim\mathrm{Tor}_0^{\mathcal{A}_B}(\mathcal{H}_{S,T,U,V})&\approx\dim\mathrm{Tor}_0^{\mathcal{A}_B}(\mathcal{H}_{S\setminus\{w\},T\cup\{w\},U,V})\otimes \delta_{d(w)}(1)\\
%&\qquad\qquad\qquad\qquad-\dim\mathrm{Tor}_0^{\mathcal{A}_B}(\mathcal{H}_{S\setminus\{w\},T,U\cup\{w\},V})\\
%&\approx(d(w)-1)\hat \sigma\dim\mathrm{Tor}_0^{\mathcal{A}_B}(\mathcal{H}_{S\setminus\{w\},T\cup\{w\},U,V})\\
%&\qquad\qquad\qquad\qquad-\dim\mathrm{Tor}_0^{\mathcal{A}_B}(\mathcal{H}_{S\setminus\{w\},T,U\cup\{w\},V})\\
%&\approx (d(w)-1)\hat \sigma\dim\mathrm{Tor}_0^{\mathcal{A}_B}(\mathcal{H}_{S\setminus\{w\},T\cup\{w\},U,V})\\
%&\qquad\qquad\qquad\qquad  -\dim\mathrm{Tor}_0^{\mathcal{A}_B}(\mathcal{H}_{S\setminus\{w\},T\cup\{w\},U,V}),
%\end{align*}} where we have used Corollary \ref{cor:replacement} in the first and third lines.
%\end{proof}

\begin{proof}[Proof of Theorem \ref{thm:upper bound}]
For the first claim, we consult the vertex explosion spectral sequence for $B$, which has the form 
\[E^2_{p,q}\cong \mathrm{Tor}_p^{\mathcal{A}_B}\left(\mathcal{H}_q(\graf_B)\right)\implies \mathcal{H}_{p+q}(\graf),\] and the map in question is simply the edge map $E^2_{0,i}\to E^\infty_{0,i}\subseteq \mathcal{H}_i(\graf)$ from the leftmost column. Its cokernel, therefore, is a subquotient of $\bigoplus_{p=1}^{|B|} E^2_{p, i-p}$ by Corollary \ref{cor:tor dimension}, so it suffices in this case, by Proposition \ref{prop:closure properties}, to show that $\mathrm{Tor}_p^{\mathcal{A}_B}(\mathcal{H}_q(\graf_B))\in o(k^{\Lambda^i_\graf-1}k!)$
for $p+q=i$ provided $p>0$. In light of the decomposition
\[\mathcal{H}_q(\graf_B)=\bigoplus_{|S|=q} \mathcal{H}_{S,\varnothing,\varnothing, W\setminus S},\] Theorem \ref{thm:technical bound} grants membership in $o(k^{\Lambda_\graf^{q}}k!)$, and stipulating that $p>0$ guarantees that $q<i$, so Lemma \ref{lem:basic inequality} implies the claim (in the case $i=1$, where the lemma does not apply, we use our assumption that $\Lambda^1_\graf>0=\Lambda^0_\graf$). The kernel, on the other hand, is the kernel of the quotient map $E^2_{0,i}\to E^\infty_{0,i}$. From the exact sequence
\[0\to \mathrm{im}\, d^r_{r,i-r+1}\to E^r_{0,i}\to E^{r+1}_{0,i}\to 0,\] the isomorphism $E^\infty_{0,i}\cong E^{i+1}_{0,i}$, and rank-nullity, it follows that this kernel has dimension \[\sum_{r=2}^{i} \dim\mathrm{im}\, d^r_{r,i,-r+1}\leq \sum_{r=2}^i \dim E^r_{r,i-r+1}\leq \sum_{r=2}^i \dim E^2_{r,i-r+1},\] so it suffices to observe, by the same argument as above, that $\mathrm{Tor}_p^{\mathcal{A}_B}(\mathcal{H}_q(\graf_B))\in o(k^{\Lambda^i_\graf-1}k!)$ for $p+q=i+1$ provided $p>1$.

The second claim follows from Lemma \ref{lem:H0} after repeated application of Lemma \ref{lem:shift explosion} and Corollaries \ref{cor:hat shift little o} and \ref{cor:move a vertex}. Modulo notation, the proof in characteristic zero is identical save that we instead invoke Corollary \ref{cor:shift little pi} in the last step.

\end{proof}

\subsection{Proof of Lemma \ref{lem:base case}} We begin with two easy reductions. First, if $b\in B$ is topologically adjacent to a univalent vertex of $\graf$, then, as in Lemma \ref{lem:shift explosion}, we have the isomorphism 
\[\mathrm{Tor}^{\mathcal{A}_B}_p(\mathcal{H}_{\varnothing,\varnothing,\varnothing,W}(\graf_B))\cong \mathrm{Tor}^{\mathcal{A}_{B\setminus\{b\}}}_p(\mathcal{H}_{\varnothing,\varnothing,\varnothing,W}(\graf_{B\setminus\{b\}})).\] Thus, by induction, we may assume that no member of $b$ has this property. In the same way, we may assume that no two elements of $B$ are adjacent.

In this case, we may write $\graf=\graf'\sqcup [0,1]^{\sqcup n}$ and $\graf_B=\graf'_B\sqcup [0,1]^{\sqcup n}$, where every component of $\graf'$ has an essential vertex. Note that, in this situation, we have $\Lambda^0_\graf=n$. From the definitions and Proposition \ref{prop:H0 description}, setting $m=|\pi_0(\graf'_B)|$, we have 
\[\mathcal{H}_{\varnothing,\varnothing,\varnothing,W}(\graf_B)\cong \mathcal{H}_{\varnothing,\varnothing,\varnothing,W}(\graf'_B)\otimes \mathcal{A}_n\cong \mathcal{S}_m\otimes\mathcal{A}_n.\] 

For the first claim of the lemma, Corollary \ref{cor:koszul tor} permits us to calculate the $\mathrm{Tor}$ objects of interest as the homology of a complex with underlying graded object $\left(\mathbb{F}(1)\oplus \mathbb{F}\right)^{\otimes B}\otimes \mathcal{S}_m\otimes \mathcal{A}_n$, which is isomorphic to a finite sum of shifts of $\mathcal{S}_m\otimes \mathcal{A}_n$. Thus, it suffices by Proposition \ref{prop:closure properties} to note that this symmetric sequence lies in $o(k^nk!)$ by Theorem \ref{thm:H0}.

For the claim regarding characteristic zero, we require the following purely algebraic result.

\begin{lemma}\label{lem:tor bound}
Let $\mathcal{M}\in \pi(k^{d+1})$ be an $\mathcal{A}_1$-module in characteristic zero. 
\begin{enumerate}
\item If $\mathrm{Tor}_0^{\mathcal{A}_1}(\mathcal{M})\in \pi(k^{d})$, then $\mathrm{Tor}_1^{\mathcal{A}_1}(\mathcal{M})\in \pi(k^{d})$.
\item If $\mathcal{A}_1$ acts trivially on $\mathcal{M}$, then $\mathrm{Tor}_p^{\mathcal{A}_1}(\mathcal{M})\in \pi(k^{d+1})$ for $p\in\{0,1\}$.
\end{enumerate}
\end{lemma}
\begin{proof}
By Corollary \ref{cor:koszul tor}, Proposition \ref{prop:function identities}, and rank-nullity, we have the equation
\[\mul\, \mathrm{Tor}_1^{\mathcal{A}_1}(\mathcal{M})=(1-\hat\sigma) \mul\,\mathcal{M}+\mul\,\mathrm{Tor}_0^{\mathcal{A}_1}(\mathcal{M}),\] 
so the first claim follows from Corollary \ref{cor:shift little pi}. On the other hand, if $\mathcal{A}_1$ acts trivially, then Corollary \ref{cor:koszul tor} gives the isomorphisms $\mathrm{Tor}_0^{\mathcal{A}_1}(\mathcal{M})\cong\mathcal{M}$ and $\mathrm{Tor}_1^{\mathcal{A}_1}(\mathcal{M})\cong\mathbb{F}(1)\otimes\mathcal{M}$, and the claim follows from our assumption on $\mathcal{M}$ and Proposition \ref{prop:function identities} and Corollary \ref{cor:shift little pi}.
\end{proof}

To proceed, recalling that $\mathcal{A}_B=\bigotimes_{b\in B}\mathcal{A}_1$, and given an ordering of $B$, we may express the $\mathrm{Tor}_p$ term of interest (up to iterated extensions) as the following iterated Tor by Corollary \ref{cor:iterated tor}:
\[\bigoplus_{{\epsilon_1+\cdots +\epsilon_{|B|}=p}} \mathrm{Tor}_{\epsilon_{|B|}}^{\mathcal{A}_1}\left(\cdots\left(\mathrm{Tor}_{\epsilon_{2}}^{\mathcal{A}_1}\left(\mathrm{Tor}_{\epsilon_{1}}^{\mathcal{A}_1}(\mathcal{S}_m\otimes \mathcal{A}_n)\right)\cdots\right)\right),\] where $\epsilon_i\in \{0,1\}$. Thus, by Proposition \ref{prop:closure properties}, it suffices to show that each summand of the above expression lies in $\pi(k^{\Delta^0_\graf})$, which we will achieve through inductive use of Lemma \ref{lem:tor bound}, with base case provided by Theorem \ref{thm:H0}. 

In order to reduce clutter, we abbreviate the summand shown above to $\mathcal{T}_\epsilon$, where $\epsilon=(\epsilon_1,\ldots, \epsilon_{|B|})$, and we define graphs $\graf_j$ recursively by setting $\graf_{|B|}=\graf$ and $\graf_j=(\graf_{j+1})_{b_{j+1}}$; thus, in particular, we have $\graf_{0}=\graf_B$. 

%We claim first that, if $|\pi_0(\graf_j)|=|\pi_0(\graf_{j+1})|$, then the copy of $\mathcal{A}_1$ indexed by $b_{j+1}$ acts trivially on $\mathcal{T}_{(\epsilon_1,\ldots, \epsilon_{j})}$. Indeed, in this case there is a sequence of distinct leaves $(\ell_1,\ldots, \ell_{2n})$ of $\graf_B$ with the following properties: $\ell_1$ lies in the same component of $\graf_B$ as one of the leaves associated to $b_{j+1}$; $\ell_{2n}$ lies in the same component as the other; for $i$ odd, $\ell_m$ and $\ell_{m+1}$ are leaves associated to $b_i$ for some $i\leq j$; and all of the $\ell_m$ lie in the same component of $\graf_j$. The claim follows.

We claim first that, if $|\pi_0(\graf_j)|=|\pi_0(\graf_{j+1})|$, then the copy of $\mathcal{A}_1$ indexed by $b_{j+1}$ acts trivially on $\mathcal{T}_{(\epsilon_1,\ldots, \epsilon_{j})}$. Indeed, in this case the two leaves associated to $b_{j+1}$ belong to the same connected component of $\graf_j$, and therefore there is a sequence of distinct leaves $(\ell_1,\ldots,\ell_{2n})$ of $\graf_B$ with the following properties: $\ell_1$ and $\ell_{2n}$ are the two leaves associated to $b_{j+1}$; for each $k\leq n$ the leaves $\ell_{2k-1}$ and $\ell_{2k}$ belong to the same component of $\graf_B$; and for each $k<n$ the leaves $\ell_{2k}$ and $\ell_{2k+1}$ are the two leaves associated to some $b_i$ for some $i\leq j$. By Proposition \ref{prop:H0 description}, the leaves $\ell_{2k-1}$ and $\ell_{2k}$ have the same action on $\mathcal{S}_m\otimes \mathcal{A}_n$. The claim follows upon noting that, for $k< n$, when we get to $\mathcal{T}_{(\epsilon_1,\ldots,\epsilon_i)}$, either we have $\epsilon_i = 0$ and identify the two actions from $\ell_{2k}$ and $\ell_{2k+1}$, or we have $\epsilon_i = 1$ and pass to a submodule where these two actions coincide. 

Thus, by Lemma \ref{lem:tor bound}, after perhaps adjusting the ordering of $B$, we may assume that $|\pi_0(\graf_j)|=|\pi_0(\graf_{j+1})|+1$ for each $j$. With this assumption, we have $\Delta^0_\graf=m+n-|B|$, so it suffices to show that $\mathcal{T}_{(\epsilon_1,\ldots, \epsilon_j)}\in \pi(k^{m+n-j})$ for each $j$. We now proceed by induction on $j$, the base case of $j=0$ being Theorem \ref{thm:H0}. For fixed $j$, by induction on $\max(i: \epsilon_i=1)$ using Lemma \ref{lem:tor bound}, we may assume that $\epsilon_i=0$ for $1\leq i\leq j$, in which case $\mathcal{T}_\epsilon$ is isomorphic to the $E^2_{0,0}$ term of the vertex explosion spectral sequence for $\graf_j$ and $\{b_1,\ldots, b_j\}$. It follows in this case that $\mathcal{T}_\epsilon\cong \mathcal{H}_0(\graf_j)$, and the claim follows from Proposition \ref{prop:H0 description} and Theorem \ref{thm:H0}, since $\Lambda^0_{\graf_j}=n$ and $\mathcal{E}^0_{\graf_j}=m-j$ by assumption.

\bibliographystyle{amsalpha}
\bibliography{references}

\end{document}